\documentclass[11pt]{article}
\usepackage[T1]{fontenc}
\usepackage[cp1250]{inputenc}
\usepackage{amsmath,amstext,amssymb,amsopn,amsthm}
\usepackage{graphicx}

\theoremstyle{plain}
\newtheorem{thm}{Theorem}[section]

\newtheorem{lem}[thm]{Lemma}
\newtheorem{prop}[thm]{Proposition}
\newtheorem{cor}[thm]{Corollary}

\theoremstyle{definition}
\newtheorem{defn}{Definition}[section]

\theoremstyle{remark}
\newtheorem*{rem*}{Remark}

\newcommand{\R}{\mathbb{R}}

\newcommand{\Z}{\mathbb{Z}}
\newcommand{\C}{\mathbb{C}}
\newcommand{\E}{\mathbb{E}}

\renewcommand{\leq}{\leqslant}
\renewcommand{\geq}{\geqslant}

\renewcommand{\leq}{\leqslant}
\renewcommand{\geq}{\geqslant}

\newcommand{\pref}[1]{(\ref{#1})}

\def\({\left(}
\def\){\right)}
\def\[{\left[}
\def\]{\right]}
\def\<{\langle}
\def\>{\rangle}

\title{Hitting half-spaces by Bessel-Brownian diffusions
\footnotetext{2000 MS Classification:
    Primary 60J65; Secondary 60J60.
    {\it Key words and phrases}: Bessel processes, Bessel kernels, Riesz kernels,
   relativistic process, stable process, Poisson kernel, Green function,
    half-spaces.  Research  supported by  Polish Ministry of Science and Higher Eduction 
    grant N N201 3731 36  }
    }
\author{T. Byczkowski, J. Ma{\l}ecki, M. Ryznar\\
 Institute of Mathematics and Computer Sciences,\\
  Wroc\l{}aw University of Technology, Poland}

\begin{document}
\maketitle
\begin{abstract}
The purpose of the paper is to find explicit formulas describing the joint distributions of the first hitting time and place for half-spaces of codimension one for a diffusion in $\R^{n+1}$, composed of one-dimensional Bessel  process  and independent $n$-dimensional  Brownian motion.  The most important argument is carried out for the two-dimensional situation. We show  that this amounts to computation of
distributions of various integral functionals with respect to a two-dimensional  process with independent  Bessel components. As a result, we provide a formula for 
the Poisson kernel of a {\it{half-space}} or of a  {\it{strip}}  for 
 the operator $(I-\Delta)^{\alpha/2}$, $0<\alpha<2$. In the case of a {\it{half-space}}, this 
 result was recently found, by different methods, in \cite{BMR:2009}. As an application of our method we also compute various 
 formulas for first hitting places for the {\it{isotropic stable L\'evy process}}.
\end{abstract}

\section{Introduction}
Bessel processes appear in many important theoretical as well as practical applications.  A systematic study, in the frame of one-dimensional diffusions, was initiated by H.P. McKean in \cite{McKeane:1960}. Bessel processes, in particular, play an 
important role in a deeper understanding of the local time of the standard Brownian motion (see Ray-Knight Theorem in e.g. \cite{RevuzYor:2005}). Their generalizations serve as various models in applications \cite{DonatiMartinYor:1997}. In our presentation, when considering Bessel processes,  we follow the exposition of \cite{RevuzYor:2005}, Ch. XI.

The purpose of the paper is to exploit a correspondence between harmonic measure of the operator in $\R^n$
\begin{equation*}
(I-\Delta)^{\alpha/2}\,,\quad 0<\alpha<2\/,
\end{equation*}
for a {\it{half-space}} or a {\it{strip}} 
and the joint distribution of the first hitting time and place for appropriate sets of codimension one for the $(n+1)$-dimensional {\it{Bessel-Brownian diffusion}}, that is, the process 
$\textbf{Y}(t) = (R_t^{(-\alpha/2)},B^{n}(t))$, where $R_t^{(-\alpha/2)}$ is one-dimensional $BES^{(-\alpha/2)}$ process and $B^{n}$ is 
independent $n$-dimensional Brownian motion.
Although we do not appeal to  Molchanov-Ostrovski representation (see \cite{Molcanov:1969}) our results exemplify their 
principle that  some jump Markov processes (here: {\it{isotropic stable}} and {\it{relativistic L\'evy processes}}) can be regarded as traces of the appropriate  
Bessel-Brownian diffusions. 
With this regard we also mention the paper \cite{Banuelos:2004} where the eigenvalue problem for the Cauchy process in $R^n$ was transformed into a kind of mixed eigenvalue problem for the Laplace operator  in $R^{n+1}$ and also the work \cite{CafSi:2007} which proposes to study $n$-dimensional non-local
operators by means of $(n + 1)$-dimensional local operators. 

The paper is organized as follows.
In Preliminaries (Section 2) we enclose some essentials about various special functions, indispensable in the sequel. 
 We  also provide a basic information about squared Bessel processes, Bessel processes and their hyperbolic and trigonometric variants, needed in the sequel. A detailed discussion of generalized Bessel processes is carried out in Appendix.
 
 In Section 3 we state our basic observation that the  $m$-harmonic measure for some regular sets $\tilde{D}\subset \R^n$ for the 
 $n$-dimensional $\alpha$-stable relativistic process can be read from the joint distribution of the first exit time and place for the set $D \subset \R^{n+1}$ (with the trace $\tilde{D}$) for our Bessel-Brownian diffusion. The proof of this, intuitively apparent result, is postponed to Appendix. 
 When $m=0$ we obtain the same statement for the 
 usual harmonic measure for the standard isotropic $\alpha$-stable process. 
 
In Section 4 we first 
 compute the joint distribution of the first hitting time and place for 
the vertical positive axis in the two-dimensional space for the process $\textbf{Y}$. Next, we
 deal with the same process $\textbf{Y}$ but this time hitting two horizontal 
half-lines $(-\infty,-1]$ and $[1,\infty)$.
This is crucial for further purposes. We  rely on  stochastic calculus; in particular, we apply appropriate random change of time and compute various integral functionals of Bessel processes and generalized Bessel processes. The resulting formulas are, in the first case,  in terms of {\it modified Bessel functions}; in the second one - in terms of {\it{spheroidal wave functions}}. In the end of this section we apply our results to provide a purely probabilistic method of computing the Poisson kernel of the interval $[-1,1]$ for the standard isotropic $\alpha$-stable process.  

In Section 5 we generalize previous results to multidimensional case. In the view of our principle from Section 3, the first part of Section 5 
  yields the formula for  the Poisson kernel of a half-space
 for the operator 
$(I-\Delta)^{\alpha/2}$, $0<\alpha <2$, 
recently found in \cite{BMR:2009}, using rather analytical methods. The second part gives  a description of the 
Poisson kernel for a strip; also for the standard isotropic $\alpha$-stable process.


%
%

\section{Preliminaries}
In this section we collect some preliminary material. For more information on modified Bessel functions, Whittaker's functions and hypergeometric functions we refer to \cite{AbramowitzStegun:1972} and \cite{Erdelyi:1955}. For questions regarding Bessel processes, stochastic differential equations and one-dimensional diffusions we refer to \cite{RevuzYor:2005} and to
\cite{IW}.  
\subsection{Special functions } 
\subsubsection*{Modified Bessel Functions}

Various potential-theoretic objects in the theory of the relativistic process are expressed
 in terms of modified Bessel functions  $I_{\vartheta}$ and $K_{\vartheta}$.
 For convenience
 we collect here basic information about these functions.

The \textit{modified Bessel function $I_{\vartheta}$ of the first kind} is defined by (see, e.g. \cite{Erdelyi:1955}, 7.2.2 (12)):
  \begin{equation}
    \label{I_definition}
     I_{\vartheta}(z) = \(\frac{z}{2}\)^\vartheta \sum_{k=0}^{\infty} \left(\frac{z}{2}\right)^{2k}\frac{1}{k!\Gamma(k+\vartheta+1)} \/, \quad    z \in \C \setminus (-\R_{+}) \/,
  \end{equation}
  where $\vartheta\in \R$.
  The \textit{modified Bessel function of the third kind} is defined by (see \cite{Erdelyi:1955}, 7.2.2 (13) and (36)):
  \begin{eqnarray}
    \label{K_definition1}
    K_\vartheta (z) &=& \frac{\pi}{2\sin(\vartheta\pi)}\left[I_{-\vartheta}(z)-I_\vartheta(z)\right]\/,\quad
    \vartheta \notin \Z \/,\\
    \label{K_definition2}
    K_n (z) &=& \lim_{\vartheta\to n}K_\vartheta(z) = (-1)^n
    \frac{1}{2}\left[\frac{\partial I_{-\vartheta}}{\partial \vartheta} -
    \frac{\partial I_{\vartheta}}{\partial \vartheta}\right]_{\vartheta=n}\/,\quad n\in \Z\/.
  \end{eqnarray}
  We will also use the following integral representations of the function $K_\vartheta(z)$ (\cite{Erdelyi:1955}, 7.11 (23) or \cite{GradsteinRyzhik:2007}, 8.432 (6)):
  \begin{eqnarray}
     \label{Macdonald}
     K_\vartheta(z) = 2^{-\vartheta-1} z^{\vartheta} \int_0^\infty e^{-t}e^{-\frac{z^2}{4t}}t^{-\vartheta-1}\,dt\/,
  \end{eqnarray}
  where $\Re(z^2)>0$, $|\arg z|<\frac{\pi}{2}$.

In the sequel we will use the asymptotic behavior of $I_\vartheta$, as a function of real variable $r$:
 \begin{eqnarray}
  I_\vartheta(r)&\cong& \frac{1}{\Gamma(\vartheta+1)}
  \left(\frac{r}{2}\right)^{\vartheta}\,,
   \quad r\to 0^+, \label{asympt_I_0}
 \end{eqnarray}
 where  $g(r) \cong f(r) $ means that the ratio of $g$ and $f$ tends to $1$. When $r\geq 1$ we have
\begin{equation} \label{asympt_I_infty}
I_\vartheta(r)= (2\pi r)^{1/2}\,e^r [1+ E(r)]\,,
\end{equation}
where $E(r)=O(r^{-1})$, $r\rightarrow \infty$.

\subsubsection*{Confluent hypergeometric function and Whittaker's functions}
The Laplace transforms of some additive functionals of Bessel processes are given in terms of the Whittaker's functions. We introduce some basic notation and properties of these functions. 

   The \textit{confluent hypergeometric function} is defined by
   \begin{eqnarray*}
      \Phi(a,b;z) &=&\sum_{k=0}^\infty \frac{(a)_k}{(c)_k}\, \frac{z^k}{k!}\/,\quad b\neq 0, -1,-2,\ldots\/,
   \end{eqnarray*}
   where $z$ is a complex variable, $a$ and $b$ are parameters. Here $(a)_k = \Gamma(a+k)\backslash \Gamma(a)$ denotes the Pocchamer symbol. For $|\arg z|<\pi$ and $b\neq 0, \pm1,\pm2,\ldots$ we define a new function
   \begin{eqnarray*}
      \Psi(a,b;z) = \frac{\Gamma(1-b)}{\Gamma(1+a-b)}\,\Phi(a,b;z) + \frac{\Gamma(c-1)}{\Gamma(a)}\, z^{1-b} \,\Phi(1+a-b,2-b;z)
   \end{eqnarray*}
   called the \textit{confluent hypergeometric function of the second kind}. The \textit{Whittaker's function of the first and the second kind} are  defined respectively by
   \begin{eqnarray}
      \label{WhittakerDefinitionM}
      M_{\kappa, \,\mu}(z) &=& z^{\mu+1/2}\,e^{-z/2}\, \Phi(1/2-\kappa+\mu,2\mu+1;z)\/, \\
      \label{WhittakerDefinitionW}
      W_{\kappa, \,\mu}(z) &=& z^{\mu+1/2}\,e^{-z/2}\, \Psi(1/2-\kappa+\mu,2\mu+1;z)\/,\quad |\arg z|<\pi. 
   \end{eqnarray}
   The confluent hypergeometric function of the second kind satisfies the relation
   \begin{eqnarray*}
       \Psi(a,b;z) &=& z^{1-b}\Psi(1+a-b,2-b;z)\/,\quad |\arg z| <\pi\/, 
   \end{eqnarray*}
   which implies that
   \begin{eqnarray}
    \label{WhittakerSymmetry} 
       W_{\kappa, \,\mu}(z) &=& W_{\kappa, \,-\mu}(z)\/.
   \end{eqnarray}
   The following asymptotic formula (\cite{AbramowitzStegun:1972}, 13.5.8 p. 508), valid for $1 <\Re (b) <2$,
   \begin{eqnarray*}
     \Psi(a,b;z) = \frac{\Gamma(b-1)}{\Gamma(a)} z^{1-b}+O(1)\/,\quad |z|\to 0
   \end{eqnarray*}
   together with the definition of the Whittaker's function give that for $0<\mu<1/2$ we have
   \begin{eqnarray}
   \label{WhittakerAsymptotic} 
       W_{\kappa,\,\mu}(z) &=& z^{\mu+1/2}\,e^{-z/2}\(\frac{\Gamma(2\mu)}{\Gamma(1/2-\kappa+\mu)} z^{-2\mu}+O(1)\)\/, \quad |z|\to 0\/.
   \end{eqnarray}
\subsubsection*{Hypergeometric function and Legendre functions}   
   For $c\neq 0,-1,-2,\ldots$ we define the \textit{hypergeometric function} by
   \begin{eqnarray*}
      \,_2F_1(a,b;c;z) = \sum_{n=0}^\infty \frac{(a)_n(b)_n}{(c)_n n!}z^n\/,\quad |z|<1\/.
   \end{eqnarray*}
   If $\Re(c-a-b)>0$ then 
   \begin{equation}
      \label{HyperOne}
      \,_2F_1(a,b;c;1) = \frac{\Gamma(c-a-b)\Gamma(c)}{\Gamma(c-a)\Gamma(c-b)}\/.
   \end{equation}
   The function $\,_2F_1(a,b;c;z)$ is a solution to the \textit{hypergeometric equation}
   \begin{equation}
      \label{HyperDE}
      z(1-z)y''(z)+[c-(a+b+1)z]y'(z) -ab y(z) = 0\/.
   \end{equation} 
   If $c$ is not an integer then the other independent solution is given by
   \begin{eqnarray}
      \label{HyperChange}
      z^{1-c}\,_2F_1(a+1-c,b+1-c;2-c;z) = z^{1-c}(1-z)^{c-a-b}\,_2F_1(1-a,1-b;2-c;z)\/. 
   \end{eqnarray}
   The \textit{Legendre functions} are solutions of \textit{Legendre's differential equation}
   \begin{equation}
      \label{LegendreDE}
      (1-z^2)y''(z)-2z y'(z) + [\nu(\nu+1)-\mu^2(1-z^2)^{-1}]y(z) = 0\/.
   \end{equation}
   Making appropriate substitutions it can be reduced to the hypergeometric equations and consequently  
   the \textit{Legendre function of the first kind} is defined by
   \begin{equation}
   \label{LegendreFirst}
      P_\nu^\mu(z) = \frac{1}{\Gamma(1-\mu)} \(\frac{z+1}{z-1}\)^{\mu/2}\,_2F_1(-\nu,\nu+1;1-\mu;\frac{1-z}{2})
   \end{equation}
   and the \textit{Legendre function of the second kind} is 
   \begin{equation}
   \label{LegendreSecond}
     Q_\nu^\mu(z) = e^{\mu i \pi} 2^{-\nu-1}\pi^{1/2}\frac{\Gamma(\nu+\mu+1)}{\Gamma(\nu+\frac32)}z^{-\nu-\mu-1}(z^2-1)^{\mu/2}\,_2F_1(\frac{\nu+\mu}{2}+1,\frac{\nu+\mu+1}{2};\nu+\frac32;\frac{1}{z^{2}})\/.
   \end{equation}
   The Wronskian of the pair $\{P_\nu^\mu(z),Q_\nu^\mu(z)\}$ is equal to
   \begin{equation}
   \label{LegendreWron} \frac{1}{1-z^2}e^{i\mu\pi}2^{2\mu}\frac{\Gamma(\frac{\nu+\mu}{2}+1)\Gamma(\frac{\nu+\mu+1}{2})}{\Gamma(1+\frac{\nu-\mu}{2})\Gamma(\frac{1+\nu-\mu}{2})}\/.
   \end{equation}

\subsection{Bessel processes}
\subsubsection*{Time change of squared Bessel processes}
The basic material  concerning  Bessel processes is taken from
\cite{RevuzYor:2005}, Ch. XI. 

We begin with a definition of a square of Bessel process:
\begin{defn} Let $\beta(t)$ be a   Brownian motion.
For every $\delta\geq 0$ and $x \geq 0$, the unique strong solution of the equation
\begin{equation*}
dZ(t) = 2\,\sqrt{|Z(t)|}d\beta(t) + \delta\,dt\,, \quad Z(0)=x\,,
\end{equation*}
is called the {\it square} of $\delta$-{\it dimensional Bessel} process started at $x$ and is denoted 
by $BESQ^{\delta}(x)$. $\delta$ is the dimension of $BESQ^{\delta}$. The square root of 
$BESQ^{\delta}(a^2)$, $a\geq0$, is called the {\it Bessel process} of dimension $\delta$ started at $a$ and is denoted by $BES^{\delta}(a)$.
\end{defn}
We introduce also the {\it index} $\nu=(\delta/2)-1$ of the corresponding process, and write
$BESQ^{(\nu)}$ instead of $BESQ^{\delta}$ if we want to use $\nu$ instead of $\delta$. The same convention applies to $BES^{(\nu)}$.

For $\nu>-1$, the semi-group $BESQ^{(\nu)}(x)$ has a transition density function 
\begin{equation}
\label{BESQ}
q_t^{(\nu)}(x,y) = \frac{1}{2}
\left(\frac{y}{x}\right)^{\nu/2}
  e^{-(x+y)/2t}I_{\nu} 
     \(\sqrt{xy}/{t}\)
\quad {\text for} \quad x>0
\end{equation}
and
\begin{equation}
\label{BESQzero}
q_t^{(\nu)}(0,y) = (2t)^{-\nu-1}t^{-(\nu+1)}\Gamma(\nu+1)^{-1}\,y^{\nu}
  e^{-y/2t} \,.
\end{equation}
$I_{\nu}$ denotes here the modified Bessel function of the first kind. 
It is well known that for $-1<\nu<0$ the point $0$ is reflecting; for $\nu\geq 0$ the point $0$ is polar.

The infinitesimal generator of $BESQ^{\delta}$ is equal on $C^2(0,\infty)$ to the operator
\begin{equation}
\label{genBESQ}
2\,x\frac{d^2}{dx^2} + \delta\frac{d}{dx}\,.
\end{equation}

If $X$ is a $BESQ^{\delta}(x)$, then for any $c>0$, the process $c^{-1}X(ct)$ is a 
 $BESQ^{\delta}(x/c)$.

Next, we begin with introducing the two dimensional process $Y=(Y_1,Y_2)$ with independent components as $Y_1$ being the squared $BESQ^{(-\alpha/2)}$ process and $Y_2$ - the standard Brownian motion. Equivalently, $Y=(Y_1,Y_2)$ satisfies the following stochastic diferential equations
\begin{equation}
\label{Y_SDE}
 \left\{
    \begin{array}{ccl}
      dY_1 &=& 2\sqrt{|Y_1|}d\tilde{B}_1 +(2-\alpha)dt\\
      dY_2 &=& d\tilde{B}_2\,,
    \end{array}\right.
\end{equation}
where 
 $(\tilde{B}_1,\tilde{B}_2)$ denote the standard two dimensional Brownian motion ($E\tilde{B}_i(t)=0$,
$E\tilde{B}^2_i(t)=t$ for $i=1,2$),
and where
 $0<\alpha<2$, $Y_1(0)=y_1\geq 0$ and $Y_2(0)=y_2$, with $y_2\in\R$.
It is well-known that $Y_1\geq 0$. Consequently, the absolute value in the first equation of \pref{Y_SDE} can be discarded. Also, for $-1< \nu<0$, the point $0$ is (instantaneously) reflecting and the process 
$Y_1$ is (pointwisely) recurrent.

Let $(B_1,B_2)$ be the standard Brownian motion in $\R^2$ independent from $(\tilde{B}_1,\tilde{B}_2)$. We consider another pair  $(X_1,X_2)$ of independent squared Bessel processes $BESQ^{(-\alpha/2)}$ defined by the following system of stochastic differential equations
\begin{equation}
\label{HS_X_SDE}
 \left\{
    \begin{array}{ccl}
      dX_1 &=& 2\sqrt{X_1}\,dB_1 +(2-\alpha)\,dt\\
      dX_2 &=& 2\sqrt{X_2}\,dB_2 +(2-\alpha)\,dt\/,
    \end{array}\right.
\end{equation} 
such that $X_1(0)=x_1$, $X_2(0)=x_2$, where $x_1,x_2\geq 0$. 

We define a function $f=(f_1,f_2):\R^2\rightarrow \R^2$ by $f(x_1,x_2)=(4x_1x_2, x_2-x_1)$. 
Let $(Z_1,Z_2) = f(X_1,X_2)$. Using It\^o formula we obtain
\begin{eqnarray*}
dZ_1 &=& 4X_2dX_1+4X_1dX_2\\
     &=& 4\/ X_2\,2\sqrt{X_1}\,d\/B_1 + 4\/X_2(2-\alpha)\/dt
    + 4\/ X_1\,2\sqrt{X_2}\,d\/B_2 + 4\/X_1(2-\alpha)\/dt\\
    &=& 4\/ \sqrt{X_1\/X_2}\,(2\sqrt{X_2}\,d\/B_1 + 2\sqrt{X_1}\,d\/B_2)
       +4(X_1+X_2)(2-\alpha)\/dt\\
    &=& 2\/ \sqrt{Z_1}\,(2\sqrt{X_2}\,d\/B_1 + 2\sqrt{X_1}\,d\/B_2)
       +4(X_1+X_2)(2-\alpha)\/dt\/,\\
dZ_2 &=& -2\sqrt{X_1}dB_1+2\sqrt{X_2}dB_2\/.
\end{eqnarray*}
Observe that if we put 
\begin{eqnarray*}
dU_1 &=& 2\sqrt{X_2}\,d\/B_1 + 2\sqrt{X_1}\,d\/B_2\/,\\
   dU_2 &=& -2\sqrt{X_1}dB_1+2\sqrt{X_2}dB_2\/,
\end{eqnarray*}
then the process $(U_1,U_2)$ is a continuous martingale starting from $0$ such that
\begin{eqnarray*}
   d\<U_1\>(t) &=& 4(X_1 + X_2)dt\/,\\
   d\<U_2\>(t) &=& 4(X_1 + X_2)dt\/,\\
   d\<U_1,U_2\>(t)  &=& 0\/.
\end{eqnarray*} 
If we define the following integral functional 
\begin{equation*}
   A_1(t) = 4\int_0^t (X_1(s)+ X_2(s))ds
\end{equation*}
and its generalized inverse function
\begin{equation*}
   \sigma_1(s) = \inf\{t>0: A_1(t)>s\}\/,
\end{equation*}
then we get (see \cite{IW} Theorem 7.3 p. 86) that the process $(U_1(\sigma_1(s)), U_2(\sigma_1(s)))$ is a standard $2$-dimensional Brownian motion $(\beta_1,\beta_2)$. Consequently we obtain
\begin{eqnarray*}
    dZ_1\circ\sigma_1 &=& 2\sqrt{|Z_1\circ\sigma_1|}d\beta_1+(2-\alpha)dt\/,\\
    dZ_2\circ\sigma_1 &=& d\beta_2\/.
\end{eqnarray*}
This means that $Y$ $\stackrel{d}{=}$ $Z\circ\sigma_1$, where $Y=(Y_1,Y_2)$ is the process defined by (\ref{Y_SDE}).

\subsubsection*{Time change of generalized Bessel processes}
In this section we consider the process $X=(X_1,X_2)$ given by the following  stochastic differential equations
 \begin{equation}
  \left\{
    \begin{array}{ccl}
      dX_1 &=& \sqrt{|1-X_1^2|}\,dB_1 -\dfrac{2-\alpha}{2}X_1\,dt\\
      dX_2 &=& \sqrt{|X_2^2-1|}\,dB_2 +\dfrac{2-\alpha}{2}|X_2|\,dt\/,
    \end{array}\right.
\end{equation} 
 such that $X_1(0)=x_1$, $X_2(0)=x_2$, where $ |x_1|\leq 1$ and $x_2\geq 1$. Here $B_1$, $B_2$ are two independent standard Brownian motions in $\R$. 
It follows  that that for  $|X_{1}(0)|\leq 1$ we have $|X_1(t)|\leq 1$ for all $t>0$. The boundary points $1$ and $(-1)$  are instantaneously reflecting.

Analogously, $X_2(t)\geq 1$, whenever $X_{2}(0)\geq 1$ and the boundary point $1$ is also instantaneously reflecting.
For justification of these statements, see Appendix.

The first process is a  version of {\it Legendre process}; the second one - of 
{\it hyperbolic Bessel process}.

 Obviously the processes $X_1$, $X_2$ are independent. Moreover, the generators of the processes are given by
 \begin{eqnarray}
    \label{singenerator}
    \mathcal{G}_1 &=& \frac{1-x^2}{2}\frac{d^2}{dx^2}-\frac{2-\alpha}{2}x\frac{d}{dx}\/,\\
    \label{coshgenerator}
    \mathcal{G}_2 &=& \frac{x^2-1}{2}\frac{d^2}{dx^2}+\frac{2-\alpha}{2}x\frac{d}{dx}\/,
 \end{eqnarray}
respectively. 

Let $h=(h_1,h_2):\R^2\longrightarrow\R^2$ by $h_1(x_1,x_2) = (1-x_1^2)(x_2^2-1)$ and $h_2(x_1,x_2) = x_1 x_2$. We define the process $Z=(Z_1,Z_2) = h(X_1,X_2)$. Using It\^o formula we get
\begin{eqnarray*}
   dZ_1 &=& -2\,X_1\,(X_2^2 -1)\,dX_1 + 2\,X_2\,(1-X_1^2)\,dX_2
-  (X_2^2 -1)\,(1-X_1^2)\,dt + (X_2^2 -1)\,(1-X_1^2)\,dt \\
     &=& 2\sqrt{|Z_1|}(-X_1\,\sqrt{X_2^2-1}\,dB_1 + X_2\, \sqrt{1-X_1^2}\,dB_2)
     + (2-\alpha) (X_2^2 - X_1^2)\,dt\/,\\
   dZ_2 &=& X_2\,dX_1+X_1\,dX_2  \\
     &=& (X_2\sqrt{1-X_1^2}\,dB_1+X_1\sqrt{X_2^2-1}\,dB_2)\/.
\end{eqnarray*}
Observe that for
\begin{eqnarray*}
   dW_1 &=& -X_1\,\sqrt{X_2^2-1}\,dB_1 + X_2\, \sqrt{1-X_1^2}\,dB_2\/,\\
   dW_2 &=& X_2\sqrt{1-X_1^2}\,dB_1+X_1\sqrt{X_2^2-1}\,dB_2\/,
\end{eqnarray*}
we have $d\<W_1\>(t) = d\<W_2\>(t) = X_2^2-X_1^2$ and $d\<W_1,W_2\>(t) = 0$. Thus for 
\begin{equation}
    \label{A_THL}
   A_2(t) = \int_0^t (X_2^2(s)-X_1^2(s))\,ds
\end{equation}
and $\sigma_2 = \inf\{t>0: A_2(t)>s\}$ we get that $(W_1(\sigma_2(t)), W_2(\sigma_2(t)))$ is a standard two dimensional Brownian motion $(\beta_1,\beta_2)$ and consequently we have
\begin{eqnarray*}
    dZ_1\circ \sigma_2 &=& 2\sqrt{Z_1\circ \sigma_2}\,d\beta_1 + (2-\alpha)dt\/,\\
    dZ_2\circ \sigma_2 &=& d\beta_2\/.
\end{eqnarray*} 
Comparing this with (\ref{Y_SDE}) gives $Y $ $\stackrel{d}{=}$ $ Z \circ \sigma_2$.

\section{Relativistic stable processes and Bessel-Brownian  diffusions}\label{rel}


 Assume that $0<\alpha<2$. A L\'evy process  $X^m(t)$, living in $\R^n$,  is called the  $\alpha${\it-stable  relativistic stable process} with parameter $m\ge 0$ if its  characteristic function is given by

\begin{equation*} 
E^0 e^{\left\langle z,X^m(t)\right\rangle} =
  e^{mt}e^{-t(|z|^2+m^{2/\alpha})^{\alpha/2}}\,,\quad  z\in \R^n.
 \end{equation*}
 For $m=0$ we obtain the {\it standard isotropic  $\alpha$-stable process}. The generator of $X^m(t)$ is given by $mI-(m^{2/\alpha} I-\Delta)^{\alpha/2}$. For more information about relativistic processes we refer the reader to \cite{CMS:1990} and \cite{Ryznar:2002}.
 
  For an open  set $\tilde{D}\subset {\R}^{n}$
   we define
   $
   \tau_{\tilde{D}}=\inf\{t> 0;\, X^m(t)\notin \tilde{D}\}\,.
  $
 {\it The}
  $\lambda$-{\it harmonic measure}  of the
  set $\tilde{D}$ is defined as
  \begin{equation*} \label{harm_def}
  P_{\tilde{D}}^{\lambda,m}(x,A)=
  E^x[\tau_{\tilde{D}}<\infty;\,\, e^{-\lambda \tau_{\tilde{D}}} {\bf{1}}_A(X^m(\tau_{\tilde{D}}))],
  \end{equation*}
  where $A$ is a Borel subset of $ {\R}^{n}$.
  In this paper we are interested only in the case $\lambda = m$ and  we will denote  
 $ P_{\tilde{D}}^{m,m}$ as $ P_{\tilde{D}}^{m}$. Note that the relativistic process  killed at an independent exponential time with expectation $1/m$ has the generator equal to $-(m^{2/\alpha} I-\Delta)^{\alpha/2}$. Therefore $ P_{\tilde{D}}^{m}$ can be regarded as the harmonic measure of $\tilde{D}$ for the operator  $-(m^{2/\alpha} I-\Delta)^{\alpha/2}$. If this harmonic measure is absolutely continuous with respect to the Lebesgue measure on  $\tilde{D}^c$ we call the corresponding density the $m$-Poisson kernel of $\tilde{D}$ and denote by $ P_{\tilde{D}}^{m}(\tilde{x}, \cdot)$.
 
 Next, let $\textbf{Y}(t) = (Y_1(t),B^{n}(t))$ be an $(n+1)$-dimensional diffusion with independent components, where 
 $Y_1(t)$ is $BES^{(-\alpha/2)}$ and  $B^{n}(t) = (B_2(t),B_3(t),\ldots,B_{n+1}(t))$ is the standard Brownian motion  in  $\R^{n}$.  
 
   The following proposition exhibits the fact that finding $m$-harmonic measures is equivalent to finding some   hitting distributions of the process $\textbf{Y}(t)$. 
  \begin{prop}\label{m_harmonic} Let $\tilde{D}\subset \R^n$ be open. Assume that $F=\R^n\setminus \tilde{D}$ has the interior cone property at every point. Let $D= \R^{n+1}\setminus (\{0\}\times F) $. Let $x=(0,\tilde {x})\in \{0\}\times \tilde{D}$ . Define $ \tau_{D}=\inf\{t> 0;\, \textbf{Y}(t)\notin D\}$ and assume that $P^x(\tau_D<\infty)$. Then for every Borel $A\subset {\R}^{n}$ we have
  $$P_{\tilde{D}}^{m}(x,A)=E^{x} [e^{-\frac {m^{2/\alpha}}2 \tau_D};  B^n(\tau_D)\in A].$$  The conclusion is valid also for $m=0$, that is for the harmonic measure for the standard  isotropic  $\alpha$-stable   process. 
 \end{prop}
 The proof of the above proposition is provided in the Appendix.


\section{Hitting distributions in $\R^2$}
We begin this section with considering two dimensional case, which is crucial for further computations in higher dimensions.

\subsection{Hitting distribution of a positive vertical  half-line in $\R^2$}
We compute here the joint distribution of the first hitting time and place of the positive vertical axis for the process $(R_t^{(-\alpha/2)},B(t))$
with independent components, where $R_t^{(-\alpha/2)}$ is a $BES^{(-\alpha/2)}$ process 
and $B(t)$ is the standard Brownian motion.  We always assume here that  $0<\alpha <2$. 
Define the following two sets
\begin{eqnarray*}
     H &=& \{(x_1,x_2)\in [0,\infty)\times [0,\infty): x_1>0\}\/,\\
   D_1 &=& \{(y_1,y_2)\in\R\times[0,\infty): (y_1=0) \wedge (y_2>0)\}^c\/.
\end{eqnarray*}
Let denote by $\tau_{D_1}$ the first exit time of the process $(R_t^{(-\alpha/2)},B(t))$
from the set $D_1$, i.e. $\tau_{D_1} = \inf \{t>0: (R_t^{(-\alpha/2)},B(t))\notin D_1\}$. Observe that the $BES^{(-\alpha/2)}(a)$ process $R_t^{(-\alpha/2)}$ hits the $0$ exactly when $BESQ^{(-\alpha/2)}(a^2)$
process $(R_t^{(-\alpha/2)})^2$ hits the $0$. Therefore, we consider, as in \pref{Y_SDE}
processes $(Y_1,Y_2)$ with $Y_1=(R^{(-\alpha/2)})^2$ and $Y_2=B$. We thus have
\begin{equation*}
   \tau_{D_1} = \inf \{t>0: Y(t)\notin D_1\}\/.
\end{equation*}     
We denote by $\tau_H$ the first exit time of the process $X=(X_1,X_2)$ with independent $BESQ^{(-\alpha/2)}$ components, as defined by \pref{HS_X_SDE}, from the set $H$, i.e.
\begin{equation*}
   \tau_{H} = \inf \{t>0: X(t)\notin H\} = \inf \{t>0: X_1(t)=0\}\/.
\end{equation*}
Recall that $f(x_1,x_2)=(4x_1x_2,x_2-x_1)$.
\begin{figure}[h!]
\begin{center}
\begin{tabular}{rcl}
\includegraphics[width=6cm]{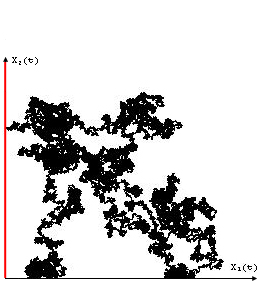}& {\Large $\stackrel{\stackrel{f}{\longrightarrow}}{\rule{0mm}{3cm}}$ }&\includegraphics[width=5.5cm]{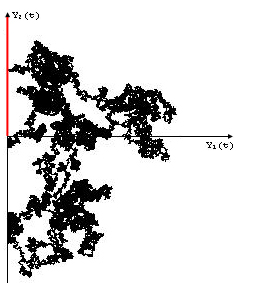}
\end{tabular}
\end{center}
\caption{Simulated paths of $X =(X_1,X_2)$ and  image $Y=(Y_1,Y_2)$ under mapping $f$.}
\end{figure}

\begin{lem}
\label{Change_HL}
  The distribution of $(\tau_{D_1}, Y(\tau_{D_1}))$ with respect to $P^{(y_1,y_2)}$ is the same as the distribution of $(A_1(\tau_H), f(X(\tau_H)))$ with respect to $(x_1,x_2)$, where $f(x_1,x_2)=(y_1,y_2)$.
\end{lem}     
\begin{proof}
  We recall that $A_1(t)= 4 \int_0^t (X_1(s) + X_2(s))\,ds$. It is easy to see that $f(H) = D_1$ and that the function $f$ is bijective on $H$. Moreover, from the time change property proved in the previous subsection we get 
  \begin{equation*}
      f(X(\sigma_1(t))) {\stackrel{d}{=}} Y(t)\/,
  \end{equation*}
  where the equality in distribution is meant for the underlying processes.
  Hence in order to prove the lemma  we may and do  assume that $f(X(\sigma_1(t))) =Y(t),\ t\ge 0$.
  We have
  \begin{eqnarray*}
      \tau_{D_1} &=& \inf \{t>0: Y(t)\notin D_1\}\\
         &=& \inf \{t>0: f(X(\sigma_1(t)))\notin D_1\}\\
         &=& \inf \{t>0: X(\sigma_1(t))\notin H\}\\
         &=& \inf \{A_1(u)>0: X(u)\notin H\}\\
         &=& A_1(\tau_H)
  \end{eqnarray*}
  and 
  \begin{equation*}
     Y(\tau_{D_1}) = f(X(\sigma_1(\tau_{D_1}))) = f(X(\sigma_1(A_1(\tau_H)))) = f(X(\tau_H)) = (0,X_2(\tau_H)).
  \end{equation*}
\end{proof}
The proof of the main result of this subsection is based on well-known formulas describing distributions of 
some integral functionals of quadratic Bessel processes, stated here in the following two lemmas. For convenience of the reader we include proofs. 
\begin{lem}
\label{besselbridge}
Let $X$ be a $BESQ^{(\nu)}$ process with $\nu>-1$. Then for $x>0$ the following holds:
\begin{eqnarray*}
E^{x}\left[\exp\(-\frac{\lambda^2}{2}\int_0^t X(s)ds\); X(t)\in dr\right]
&=&
 \left(\frac{r}{x}\right)^{\nu/2}
 \frac{\lambda }{2\sinh (t\lambda)}\, e^{-\frac{\lambda(x+r)}{2}\coth (t\lambda)}I_{\nu} 
     \(\frac{\sqrt{xr}\lambda}{\sinh (t\lambda)}\)\/.
\end{eqnarray*}
For $x=0$ we obtain
\begin{eqnarray*}
E^{0}\left[\exp\(-\frac{\lambda^2}{2}\int_0^t X(s)ds\); X(t)\in dr\right]
&=&\frac{1}{\Gamma(1-\nu)}
 \frac{r^{\nu}\lambda^{\nu+1} }{(2\sinh (t\lambda))^{\nu+1}}\, 
 e^{-\frac{\lambda r}{2}\coth (t\lambda)}\/.
\end{eqnarray*}

\end{lem}
\begin{proof}
The above follows directly from the formula for the distribution of the corresponding integral functional for the quadratic Bessel Bridge $BESQ^{(\nu)}_1(x,y)$
which can be found, e.g. in \cite{RevuzYor:2005}, Ch. XI, Corollary 3.3, p. 465:
\begin{eqnarray*}
Q^{(\nu),1}_{x,y}\left[\exp\(-\frac{\lambda^2}2\int_0^1 X(s)ds\)\right]
&=&
\frac{\lambda }{\sinh (\lambda)} \exp{\left[\frac{(x+r)}{2}(1-\lambda\coth (\lambda))\right]}
\frac{I_{\nu}\(\frac{\sqrt{xr}\lambda}{\sinh (\lambda)}\)}
{I_{\nu}\(\sqrt{xr}\)}\/,
\end{eqnarray*}
where $Q^{(\nu),1}_{x,y}$ is the distribution of the quadratic Bessel Bridge  $BESQ^{(\nu)}_1(x,y)$.

The above formula, the scaling properties of $BESQ^{(\nu)}$ processes and the formula 
\pref{BESQ} for the density of $BESQ^{(\nu)}$ give the first formula of the Lemma.
The second formula is obtained from the first one by limiting procedure as  $b \to 0 $, taking into account \pref{asympt_I_0}
\end{proof}

\begin{lem}
\label{BESQhitting}
Let $X$ be a $BESQ^{(\nu)}$ process with $\nu>-1$. Define $\tau_b = \inf \{s: X(s) = b\}$.
Then for $x>b>0$ we have

\begin{equation*}
 u_{\gamma}(x) =       E^x[\exp(-\gamma\tau_b-\frac{\lambda^2}{2}\int_0^{\tau_b}X(s)\,ds)] = \frac{x^{-(\nu+1)/2}W_{-\gamma/2\lambda,\,|\nu|/2}(\lambda x)}{b^{-(\nu+1)/2}W_{-\gamma/2\lambda,\,|\nu|/2}(\lambda b)}\,.
     \end{equation*} 
For $-1<\nu<0$ and $b=0$ we obtain for $x>0$

\begin{equation}
   \label{w_HL0}
     E^x[\exp(-\gamma\tau_0-\frac{\lambda^2}{2}\int_0^{\tau_0}X(s)\,ds)] = \frac{\Gamma(\frac{\nu+1+\lambda/2}{2})}{\lambda^{(\nu+1)/2}\Gamma(|\nu|)}x^{-(\nu+1)/2}W_{-\gamma/2\lambda,\,|\nu|/2}(\lambda x)\/.
\end{equation}

\end{lem}
\begin{proof}
By the form of the generator of the process $X$ \pref{genBESQ} and the general theory of
Feynman-Kac semigroups \cite{ChungZhao:1995} we infer that the function $u_{\gamma}(x)$
satisfies for $x>b$ the following diferential equation
\begin{equation}
\label{Whittaker}
2\,x\frac{d^2u_{\gamma}(x)}{dx^2} + 2(\nu + 1)\frac{du_{\gamma}(x)}{dx}
-((\lambda^2 x)/2 + \gamma)\,u_{\gamma}(x)=0\,.
\end{equation}
The two linearly independent solutions of \pref{Whittaker} are of the form
\begin{equation*}
\psi(x)=x^{-(\nu+1)/2} M_{-\gamma/2\lambda,|\nu|/2}(\lambda x)
\quad {\textrm{ and }} \quad
\phi(x)=x^{-(\nu+1)/2} W_{-\gamma/2\lambda,|\nu|/2}(\lambda x)\,,
\end{equation*}
where $M_{-\gamma/2\lambda,|\nu|/2}$ and $W_{-\gamma/2\lambda,|\nu|/2}$ are Whittaker's functions defined in (\ref{WhittakerDefinitionM}) and (\ref{WhittakerDefinitionW}). Taking into account the fact that the function $u_{\gamma}(x)$ is bounded and $u_{\gamma}(b)=1$ we obtain 
the first formula. 
The second formula follows from the first one and
the asymptotic formula (\ref{WhittakerAsymptotic}) which  gives
\begin{eqnarray*}
    b^{-(\nu+1)/2}W_{-\gamma/2\lambda,\,|\nu|/2}(\lambda b) 
    &=&\lambda^{(1-\nu)/2} b^{-\nu}
    e^{-\lambda b/2}\(\frac{\Gamma(|\nu|)}  {\Gamma(\frac{\nu+1+\lambda/2}{2})}\lambda^{\nu}b^{\nu}+O(1)\)\/.
\end{eqnarray*}     
Hence, we obtain
\begin{eqnarray*}
   \lim_{b\to 0} b^{-(\nu+1)/2}W_{-\gamma/2\lambda,\,|\nu|/2}(\lambda b) &=& \lambda^{(\nu+1)/2}\frac{\Gamma(|\nu|)}{\Gamma(\frac{\nu+1+\lambda/2}{2})}\/,
\end{eqnarray*}
which proves the second formula.

\end{proof}
\begin{thm} For $(R_0^{(-\alpha/2)},B(0))=(z_1,z_2) \in D_1$, $z_1>0$ and $r>0$ we have
\label{THM_HL2}
  \begin{eqnarray*}
  \nonumber
  &&   E^{(z_1,z_2)}\lefteqn{\[e^{-\frac{\lambda^2}{2}\tau_{D_1}};B(\tau_{D_1})\in dr\] =}\\
  && \frac{(|z|+z_2)^{\frac{\alpha}{4}}(|z|-z_2)^{\frac{\alpha}{2}}}{2^{\frac{3\alpha}{4}}\Gamma(\frac{\alpha}{2})r^{\alpha/4}}\int_{\lambda}^\infty e^{-(|z|+r)s}(s^2-\lambda^2)^{\alpha/4}I_{-\alpha/2} 
     \(\sqrt{2r}\sqrt{|z|+z_2}\,\sqrt{s^2-\lambda^2}\)ds\/.
  \end{eqnarray*}
  
%
%
For $z_1=0$ and $z_2=u<0$ we get
  \begin{eqnarray}
  \label{Form01_HL}
    E^{(0,u)}\[e^{-\frac{\lambda^2}{2}\tau_{D_1}};B(\tau_{D_1})\in dr\] 
    &=& \frac{\sin(\pi\alpha/2)}{\pi} \(\frac{-u}{r}\)^{\alpha/2}\frac{e^{-\lambda(r-u)}}{r-u}\/.
  \end{eqnarray}
  \end{thm}     
\begin{proof}     
  Let $\phi\in C_c^\infty(0,\infty)$  and $\lambda>0$. From Lemma \ref{Change_HL} we obtain that
  \begin{equation*}
     E^{(y_1,y_2)}\[e^{-\frac{\lambda^2}{2}\tau_{D_1}}\phi(Y(\tau_{D_1}))\] = E^{(x_1,x_2)}[e^{-\frac{\lambda^2}{2} A_1(\tau_H)}\phi(X_2(\tau_H))]\/.
  \end{equation*}
 It is convenient to carry out our computations for $\lambda/2$ instead of $\lambda$. From the independence of the processes  $X_1$
  and $X_2$ and the fact that $\tau_H$ is determined only by $X_1$ we get that
 \begin{eqnarray*}
    E^{(x_1,x_2)}[\exp(-\frac{\lambda^2}{8}\lefteqn{ A_1(\tau_H))\phi(X_2(\tau_H))]= E^{(x_1,x_2)}[\exp(-\frac{\lambda^2}{2} \int_0^{\tau_H}(X_1(s)+X_2(s))ds)\phi(X_2(\tau_H))]}\\
    &=& E^{x_1}[\exp(-\frac{\lambda^2}{2}\int_0^{\tau_H}X_1(s)ds)\left[E^{x_2}\exp(-\frac{\lambda^2}{2}\int_0^{t}X_2(s)ds)\phi(X_2(t))\right]_{t=\tau_H}]\\
    &=& E^{x_1}[\exp(-\frac{\lambda^2}{2}\int_0^{\tau_H}X_1(s)ds)\int_0^\infty\psi(\tau_H,x_2,r)\phi(r)dr]\\
     &=&
 \int_0^\infty\int_0^\infty w(t,x_1) \psi(t,x_2,r)\phi(r)dr dt\/,
 \end{eqnarray*}  
 where
 \begin{eqnarray*}
 \psi(t,x,r) &=&E^{x}\left[\exp\(-\frac{\lambda^2}2\int_0^t X_2(s)ds\); X_2(t)\in dr\right]
 \end{eqnarray*}
     and
     \begin{eqnarray*}
  w(t,x) &=&E^x\[\exp(-\frac{\lambda^2}{2}\int_0^{\tau_H}X_1(s)\,ds)\,; \tau_H \in dt\]\/.
  \end{eqnarray*}
  From Lemma \ref{besselbridge}, putting $\nu=-\alpha/2$, we obtain
  \begin{equation*}
     \psi(t,x,r) =
     x^{\alpha/4}\frac{\lambda r^{-\alpha/4}}{2\sinh (t\lambda)} e^{-\lambda\frac{(x+r)\coth (t\lambda)}{2}}I_{-\alpha/2} 
     \(\frac{\sqrt{xr}\lambda}{\sinh (t\lambda)}\)\/.
  \end{equation*}
  
To evaluate $w(t,x)$ we apply the formula \pref{w_HL0} from Lemma \ref{BESQhitting} putting again $\nu = -\alpha/2$. We obtain

%
\begin{equation}
   \label{w_HL}
     E^x[\exp(-\gamma\tau_H-\frac{\lambda^2}{2}\int_0^{\tau_H}X_1(s)\,ds)] = \frac{\Gamma(\frac12-\frac{\alpha}{4}+\frac{\lambda}{4})}{\lambda^{1/2-\alpha/4}\Gamma(\alpha/2)}x^{\alpha/4-1/2}W_{-\gamma/2\lambda,\,\alpha/4}(\lambda x)\/.
\end{equation}
Now we invert the formula (\ref{w_HL}) with respect to $\gamma$. Using (\ref{WhittakerSymmetry}) and the formula 25 p. 651 from \cite{BorodinSalminen:2002} we obtain that
\begin{equation*}
  w(t,x)=E^x\[\exp(-\frac{\lambda^2}{2}\int_0^{\tau_H}X_1(s)\,ds)\,; \tau \in dt\]
  = \frac{x^{\alpha/2}\lambda^{1+\alpha/2}}{2^{\alpha/2}\Gamma(\alpha/2)\sinh^{1+\alpha/2}(t\lambda)}\,
  \exp \(-\frac{x\lambda\cosh (t \lambda)}{2 \sinh (t \lambda)}\) dt\/.
  \end{equation*}  
Combining all the above we get
  \begin{eqnarray*}
    P^{(x_1,x_2)}(\lambda,r) &=& \int_0^\infty w(t,x_1) \psi(t,x_2,r)dt\\
    &=& \frac{x_1^{\frac{\alpha}{2}}x_2^{\frac{\alpha}{4}}\lambda^{2+\frac{\alpha}{2}}}{2^{1+\frac{\alpha}{2}}\Gamma(\frac{\alpha}{2})}\int_0^\infty \frac{r^{-\alpha/4}}{\sinh^{2+\frac{\alpha}{2}}(t\lambda)}\exp\(-\lambda\frac{(x_1+x_2+r)\coth (t\lambda)}{2}\)I_{-\frac{\alpha}{2}} 
     \(\frac{\sqrt{rx_2}\lambda }{\sinh (t\lambda)}\)dt\/.
  \end{eqnarray*}
  After substituting $\lambda\coth(t\lambda) = 2s$ we get
  \begin{eqnarray*} \frac{(x_2)^{\frac{\alpha}{4}}(x_1)^{\frac{\alpha}{2}}}{2^{\frac{\alpha}{2}}\Gamma(\frac{\alpha}{2})r^{\alpha/4}}\int_{\lambda/2}^\infty e^{-(x_1+x_2+r)s}((2s)^2-\lambda^2)^{\alpha/4}I_{-\alpha/2} 
     \(\sqrt{rx_2}\sqrt{(2s)^2-\lambda^2}\)ds\/.
  \end{eqnarray*}
  Now taking $2\lambda$ instead $\lambda$ we have 
   \begin{eqnarray*}
&&E^{(x_1,x_2)}[\exp(-\frac{\lambda^2}{2} A_1(\tau_H))\phi(X_2(\tau_H))]\\&=& \frac{(x_2)^{\frac{\alpha}{4}}(x_1)^{\frac{\alpha}{2}}}{\Gamma(\frac{\alpha}{2})r^{\alpha/4}}\int_{\lambda}^\infty e^{-(x_1+x_2+r)s}(s^2-\lambda^2)^{\alpha/4}I_{-\alpha/2} 
     \(2\sqrt{rx_2}\sqrt{s^2-\lambda^2}\)ds\/.
     \end{eqnarray*}
  
  Coming back to initial variables $(y_1 ,y_2)=(4x_1x_2,x_2-x_1)$ and 
  $(y_1,y_2)=(z_1^2,z_2)$ we obtain 
  \begin{equation*}
  x_1+x_2 = \sqrt{y_1+y^2_1}=|z|\,,\quad x_1=(|z|-z_2)/2 \quad {\text{and}} \quad x_2=(|z|+z_2)/2\,.
  \end{equation*}
  Finally
  \begin{eqnarray*}
   &&  E^{(z_1,z_2)}\lefteqn{\[e^{-\frac{\lambda^2}{2}\tau_{D_1}};B(\tau_{D_1})\in dr\] =}\\
  && \frac{(|z|+z_2)^{\frac{\alpha}{4}}(|z|-z_2)^{\frac{\alpha}{2}}}{2^{\frac{3\alpha}{4}}\Gamma(\frac{\alpha}{2})r^{\alpha/4}}\int_{\lambda}^\infty e^{-(|z|+r)s}(s^2-\lambda^2)^{\alpha/4}I_{-\alpha/2} 
     \(\sqrt{2r}\sqrt{|z|+z_2}\sqrt{s^2-\lambda^2}\)ds\/.
  \end{eqnarray*}
  From (\ref{asympt_I_0}) we get that
  \begin{equation*}
      \lim_{z\to 0}z^\nu I_{-\nu}(xz) = \frac{2^\nu}{\Gamma(1-\nu)x^\nu}\/.
  \end{equation*}
  Consequently, for $z_1=0$, $z_2=u<0$ we get
  \begin{eqnarray*}
    E^{(0,u)}\[e^{-\frac{\lambda^2}{2}\tau_{D_1}};B(\tau_{D_1})\in dr\] &=& \frac{(-2u)^{\frac{\alpha}{2}}2^{\alpha/2}}{2^{\frac{3\alpha}{4}}\Gamma(\frac{\alpha}{2})\Gamma(1-\frac{\alpha}{2})2^{\alpha/4}r^{\alpha/2}}\int_{\lambda}^\infty e^{-(r-u)s}ds\\
    &=& \frac{\sin(\pi\alpha/2)}{\pi} \(\frac{-u}{r}\)^{\alpha/2}\frac{e^{-\lambda(r-u)}}{r-u}\/.
  \end{eqnarray*}
\end{proof}  
\subsection{Hitting distribution of two half-lines in $\R^2$}\label{odcinek}
In this section we are interested in finding the distribution of the first hitting time and place of two half-lines for the process $(R_t^{(-\alpha/2)},B(t))$. 

Define the following two sets
\begin{eqnarray*}
     G &=& \{(x_1,x_2)\in [-1,1]\times [1,\infty): -1<x_1<1\}\/,\\
   C_1 &=& \{(y_1,y_2)\in\R\times[0,\infty): (y_1=0) \wedge (|y_2|>1)\}^c\/.
\end{eqnarray*}
and denote by $\tau_{C_1}$ the first exit time of the process $(R_t^{(-\alpha/2)},B(t))$ from the set $C_1$, i.e. $\tau_{C_1} = \inf \{t>0: (R_t^{(-\alpha/2)},B(t))\notin C_1\}$. As previously we can consider
the processes $Y=(Y_1,Y_2)$ defined in \pref{Y_SDE} with $Y_1=(R^{(-\alpha/2)})^2$ and $Y_2=B$. We thus have
\begin{equation*}
   \tau_{C_1} = \inf \{t>0: Y(t)\notin C_1\}\/.
\end{equation*}     
Denote by $\tau_G$ the first exit time of the process $X=(X_1,X_2)$ with independent $BESQ^{(-\alpha/2)}$ components, as defined by \pref{HS_X_SDE}, from the set $G$, i.e.
\begin{equation*}
   \tau_{G} = \inf \{t>0: X(t)\notin G\} = \inf \{t>0: |X_1(t)|=1\}\/.
\end{equation*}
The function $h$ was defined by $h_1(x_1,x_2) = (1-x_1^2)(x_2^2-1)$.

\begin{figure}[h]
\begin{center}
\begin{tabular}{rcl}
\includegraphics[width=6cm]{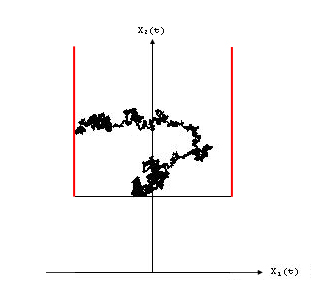}& $\stackrel{\stackrel{h}{\longrightarrow}}{\rule{0mm}{3cm}}$ &\includegraphics[width=6cm]{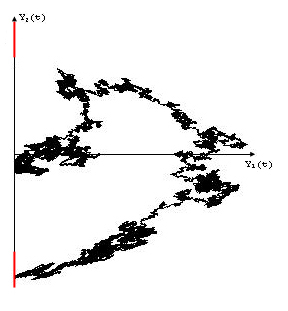}
\end{tabular}
\end{center}
\caption{Simulated paths of $X=(X_1,X_2)$ and image $Y=(Y_1,Y_2)$ under mapping $h$.}
\end{figure}
Using the same arguments as in Lemma \ref{Change_HL} we get the equality between distributions of $(\tau_{C_1}, Y(\tau_{C_1}))$ and $(A_2(\tau_G), h(X(\tau_G)))$, where in this case the integral functional $A_2(t)$ is given by
\begin{eqnarray*}
   A_2(t) &=& \int_0^t (X_2^2(s)-X_1^2(s))ds\\
     &=& \int_0^t (q_1(X_1(s))+q_2(X_2(s)))ds\/,
\end{eqnarray*}
where $q_1(x) = 1-x^2$ and $q_2(x)=x^2-1$. Notice that $q_1(X_1(s))\geq 0$ and $q_2(X_2(s))\geq 0$  almost surely. 
\begin{lem}
\label{Change_THL}
  The distribution of $(\tau_{C_1}, Y(\tau_{C_1}))$ with respect to $P^{(y_1,y_2)}$ is the same as the distribution of $(A_2(\tau_G), h(X(\tau_G)))$ with respect to $(x_1,x_2)$, where $h(x_1,x_2)=(y_1,y_2)$.
\end{lem}    
Now we state the main theorem of the section. We give a representation of the density of $(\tau_{C_1}, Y(\tau_{C_1}))$ in terms of spheroidal wave functions (see (\ref{SpheroidalEquation})). 

\begin{thm} For $(R_0^{(-\alpha/2)},B(0))=(z_1,z_2)\in C_1$ and $r>x_2\geq1$ we have
\label{THM_THL2}
  \begin{eqnarray}
  \label{Interval_formula}
     \E^{(z_1,z_2)}\[e^{-\frac{\lambda^2}{2}\tau_{C_1}};B(\tau_{C_1})\in dr\] =\frac{1}{(r^2-1)^{\alpha/2}}\frac{1}{2\pi i} \int_{\alpha/4-i\infty}^{\alpha/4+i\infty}m_{-\vartheta,\,\lambda}(x_1)\, w_\lambda(\vartheta)\phi^{\uparrow}_{\vartheta,\,\lambda}(x_2)\phi^{\downarrow}_{\vartheta,\,\lambda}(r) d\vartheta\/,
  \end{eqnarray}
  where $x_1 = \frac{1}{2}\(\sqrt{z_1^2+(z_2+1)^2}+\sqrt{z_1^2+(z_2-1)^2}\)$, $x_2 = \frac{1}{2}\(\sqrt{z_1^2+(z_2+1)^2}-\sqrt{z_1^2+(z_2-1)^2}\)$
   and the function $m_{\vartheta,\,\lambda}(\cdot)$ is the solution of the following differential equation
   \begin{eqnarray}
       \label{IntervalFirstEq}
      (1-x^2)y''(x)-(2-\alpha) x\, y'(x)-(\lambda^2(1-x^2)+2\vartheta)y(x) &=&0\/, \quad |x|<1\/,
   \end{eqnarray}
   with boundary conditions $m_{\vartheta,\,\lambda}(-1)=0$, $m_{\vartheta,\,\lambda}(1)=1$. The functions $\phi^{\uparrow}_{\vartheta,\,\lambda}(\cdot)$, $\phi^{\downarrow}_{\vartheta,\,\lambda}(\cdot)$ are respectively increasing and decreasing independent positive solutions of the differential equation 
   \begin{eqnarray}
        \label{IntervalSecondEq}
      (r^2-1)y''(r) + (2-\alpha)r\,y'(r) - (\lambda^2(r^2-1)+2\vartheta)y(r) &=&0\/,\quad r>1\/.
   \end{eqnarray}
   satisfying $\lim_{x\to1+}\phi^{\uparrow}_{\vartheta,\,\lambda}(x) =0$, $ \lim_{x\to\infty}\phi^{\downarrow}_{\vartheta,\,\lambda}(x) = 0$
   and
   \begin{eqnarray*}
      w_\lambda(\vartheta) = \frac{2}{(1-r^2)^{\alpha/2-1}W\{\phi^{\uparrow}_{\vartheta,\,\lambda},\phi^{\downarrow}_{\vartheta,\,\lambda}\}(r)}\/.
   \end{eqnarray*}
  \end{thm}    

\begin{proof}
   Let $\phi\in C_c^\infty(1,\infty)$ and $\lambda>0$. According to Lemma \ref{Change_THL} we have
   \begin{equation}
      \label{equation_01}
      \E^{(y_1,y_2)}\[e^{-\frac{\lambda^2}{2}\tau_{C_1}} \phi(Y(\tau_{C_1}))\] = \E^{(x_1,x_2)}\[e^{-\frac{\lambda^2}{2}A_2(\tau_{G})}\phi(h(X(\tau_G)))\]\/.
   \end{equation} 
   We define the following functions 
   \begin{eqnarray*}
      \Psi_\lambda(t,x,r) &=& \E^x\[\exp({-\frac{\lambda^2}{2}\int_0^t q_2(X_2(s))ds});X_2(t) \in dr\]\/,\quad x\geq1\/, t\geq 0\/, r>1\/,
   \end{eqnarray*}
   and
   \begin{eqnarray*}
      W_\lambda(t,x) &=& \E^x\[\exp({-\frac{\lambda^2}{2}\int_0^{\tau_G} q_1(X_1(s))ds}); X_1(\tau_G) = 1; \tau_G \in dt\]\/,\quad |x|\leq 1\/, t>0\/.
   \end{eqnarray*}
    Moreover, for every $\vartheta>0$, we define its Laplace transforms with respect to the variable $t$ by
   \begin{eqnarray*}
      m_{\vartheta,\,\lambda}(x) &=& \int_0^\infty e^{-\vartheta t} W_\lambda(t,x)\,dt\/,\\
      M_{\vartheta,\,\lambda}(x,r) &=& \int_0^\infty e^{-\vartheta t} \Psi_\lambda(t,x,r)\,dt\/.
   \end{eqnarray*} 
   Observe that
   \begin{eqnarray*}
      m_{\vartheta,\,\lambda} (x) &=& \int_0^\infty e^{-\vartheta t} W_\lambda(t,x)\,dt\\
      &=& \E^x\[X_1(\tau_G)=1; \exp(-\vartheta \tau_G)\exp(-\frac{\lambda^2}{2}\int_0^{\tau_G}g_1(X_1(s))ds)\]\\
      &=& \E^x\[X_1(\tau_G)=1; \exp(-\int_0^{\tau_G}g_1^*(X_1(s))ds)\]\/,
   \end{eqnarray*}
   where $q_1^*(x) = \frac{\lambda^2}{2}(1-x^2)+\vartheta$. 
   Using the Schr\"odinger equation (see \cite{ChungZhao:1995} Theorem 9.10) we get that
   \begin{eqnarray*}
       (\mathcal{G}_1m_{\vartheta,\,\lambda})(x) - q_1^*(x)m_{\vartheta,\,\lambda}(x) &=& 0\\
       (1-x^2)m_{\vartheta,\,\lambda}''(x)-(2-\alpha) x\, m_{\vartheta,\,\lambda}'(x)-(\lambda^2(1-x^2)+2\vartheta)m_{\vartheta,\,\lambda}(x) &=&0 \/.
   \end{eqnarray*}\
   For $x=-1$ and $x=1$ we have $\tau_G= 0$ and consequently $m_{\vartheta,\,\lambda}(-1)=0$ and $m_{\vartheta,\,\lambda}(1)=1$.
   
   We have
   \begin{eqnarray*}
      M_{\vartheta,\,\lambda}(x,r) &=& \int_0^\infty e^{-\vartheta t} \Psi_\lambda(t,x,r)\,dt\\
           &=& \int_0^\infty e^{-\vartheta t} \E^x\[\exp({-\frac{\lambda^2}{2}\int_0^t q_2(X_2(s))ds});X_2(t) \in dr\]\/.
   \end{eqnarray*}
      Observe that the function 
   \begin{equation*}
      p(t,x,r) = \E^x\[\exp({-\frac{\lambda^2}{2}\int_0^t q_2(X_2(s))ds});X_2(t) \in dr\]\/, \quad t>0\/;\quad x,r\geq 1\/,
   \end{equation*}
   is the transition density of one-dimensional diffusion with the generator
   \begin{equation*}
       \mathcal{L}_2u = \frac12 (x^2-1) \dfrac{d^2u}{dx^2}+\frac{2-\alpha}{2}x\dfrac{du}{dx^2}-\frac{\lambda^2}{2}(x^2-1)u\/.
   \end{equation*}
   From the general theory (see \cite{BorodinSalminen:2002},  Ch.\,II) we have that the $\vartheta$-Green function of the diffusion is given by
   \begin{eqnarray}
       \label{Diffusion_Green}
      G_\vartheta(x,r) = \int_0^\infty e^{-\vartheta t} p(t,x,r)\,dt = c_\lambda(\vartheta)\cdot \phi^{\uparrow}_{\vartheta,\,\lambda}(x)\phi^{\downarrow}_{\vartheta,\,\lambda}(r)\/,\quad x<r\/,
   \end{eqnarray}
   where $\phi^{\uparrow}_{\vartheta,\,\lambda}(x)$, $\phi^{\downarrow}_{\vartheta,\,\lambda}(x)$ are positive solutions of the differential equation
   \begin{equation*}
      \mathcal{L}_2 u -\vartheta u = 0
   \end{equation*}  
   such that $\phi^{\uparrow}_{\vartheta,\,\lambda}(x)$ is increasing and $\phi^{\downarrow}_{\vartheta,\,\lambda}(x)$ is decreasing and they satisfy the boundary conditions $\lim_{x\to1+}\phi^{\uparrow}_{\vartheta,\,\lambda}(x) =0$, $ \lim_{x\to\infty}\phi^{\downarrow}_{\vartheta,\,\lambda}(x) = 0$. The Wronskian of the pair $(\phi^{\uparrow}_{\vartheta,\,\lambda}(x),\phi^{\downarrow}_{\vartheta,\,\lambda}(x))$ is given by $W\{\phi^{\uparrow}_{\vartheta,\,\lambda},\phi^{\downarrow}_{\vartheta,\,\lambda}\}=\phi^{\uparrow}_{\vartheta,\,\lambda}(x)\frac{d}{dx}\phi^{\downarrow}_{\vartheta,\,\lambda}(x)-\frac{d}{dx}\phi^{\uparrow}_{\vartheta,\,\lambda}(x)\phi^{\downarrow}_{\vartheta,\,\lambda}(x)$ and the function  $c_\lambda(\vartheta)$ is given by
   \begin{eqnarray*}
       c_{\lambda}(\vartheta) &=&w_\lambda(\vartheta)m(r)\\
       &=& \frac{s'(r)}{W\{\phi^{\uparrow}_{\vartheta,\,\lambda},\phi^{\downarrow}_{\vartheta,\,\lambda}\}(r)}m(r)\/.\\
       &=& w_\lambda(\vartheta) \frac{1}{(r^2-1)^{\alpha/2}}\/,
   \end{eqnarray*}
   where the function
   \begin{eqnarray*}
      w_\lambda(\vartheta) &=& \frac{2}{(r^2-1)^{\alpha/2-1}W\{\phi^{\uparrow}_{\vartheta,\,\lambda},\phi^{\downarrow}_{\vartheta,\,\lambda}\}(r)}
   \end{eqnarray*}
   does not depend on $r$.
   
   Recall that $\tau_G$ depends only on $X_1$ and the processes $X_1$ and $X_2$ are independent. Thus we easily obtain that the expression in the right hand side of (\ref{equation_01}) is equal to
   \begin{eqnarray*}
        \E^{x_1}[\exp({-\frac{\lambda^2}{2}\int_0^{\tau_G} q_1(X_1(s))ds})\lefteqn{ \E^{x_2}[\exp({-\frac{\lambda^2}{2}\int_0^t q_2(X_2(s))ds})\phi(X_2(t))]_{\tau_G=1,\,X_1(\tau_G)=1}] =}\\
        &=&  \int_0^\infty W_\lambda(t,x_1)\int_1^\infty \phi(r)\Psi_\lambda(t,x_2,r)\,dr\,dt\\
        &=&  \int_1^\infty \phi(r) \int_0^\infty W_\lambda(t,x_1)\Psi_\lambda(t,x_2,r)dt\,dr\/.
   \end{eqnarray*}
   Observe that
   \begin{eqnarray*}
       |m_{\vartheta,\,\lambda}(x)| &\leq& \int_0^\infty e^{-\Re(\vartheta) t} W_\lambda(t,x)\,dt \leq \int_0^\infty e^{-\Re(\vartheta) t} W_0(t,x)\,dt = m_{\Re(\vartheta),\,0}(x)
   \end{eqnarray*}
   and the integral defining $m_{\Re(\vartheta),\,0}(x)$ is finite for every $\vartheta \in \C$ such that $\Re(\vartheta)>-\alpha/2$ (see (\ref{m_zero})). Consequently, using Fubini's theorem and the formula for the inverse Laplace transform we get
   \begin{eqnarray*}
      \int_0^\infty W_\lambda(t,x_1)\Psi_\lambda(t,x_2,r)dt &=& \int_0^\infty W_\lambda(t,x_1) \(\frac{1}{2\pi i} \int_{\alpha/4-i\infty}^{\alpha/4+i\infty} e^{\vartheta t}M_{\vartheta,\,\lambda}(x_2,r)\,d\vartheta\) dt\\
      &=& \frac{1}{2\pi i} \int_{\alpha/4-i\infty}^{\alpha/4+i\infty}\( \int_0^\infty e^{\vartheta t} W_\lambda(t,x_1)\,dt\) M_{\vartheta,\,\lambda}(x_2,r)\,d\vartheta\\
      &=& \frac{1}{2\pi i} \int_{\alpha/4-i\infty}^{\alpha/4+i\infty}m_{\vartheta,\,\lambda}(x_1) M_{\vartheta,\,\lambda}(x_2,r)\,d\vartheta\/.
   \end{eqnarray*}
   Coming back to initial variables $y_1=(1-x_1^2)(x_2^2-1)$, $y_2=x_1x_2$ and $(y_1,y_2)=(z_1^2,z_2)$ we have
   $(x_2-x_1)^2 = y_1+(y_2-1)^2$,  $(x_1+x_2)^2 = y_1+(y_2+1)^2$ and $x_2\geq x_1$.
   Finally
   \begin{eqnarray*}
      x_1 &=& \frac{1}{2}\(\sqrt{z_1^2+(z_2+1)^2}+\sqrt{z_1^2+(z_2-1)^2}\)\/,\\
      x_2 &=& \frac{1}{2}\(\sqrt{z_1^2+(z_2+1)^2}-\sqrt{z_1^2+(z_2-1)^2}\)
   \end{eqnarray*}
   and 
   \begin{eqnarray*}
     \E^{(z_1,z_2)}\[e^{-\frac{\lambda^2}{2}\tau_{C_1}};B(\tau_{C_1})\in dr\] &=&
     \frac{1}{2\pi i} \int_{\alpha/4-i\infty}^{\alpha/4+i\infty}m_{-\vartheta,\,\lambda}(x_1)M_{\vartheta,\,\lambda}(x_2,r) d\vartheta\\     &=& \frac{1}{(r^2-1)^{\alpha/2}}\frac{1}{2\pi i} \int_{\alpha/4-i\infty}^{\alpha/4+i\infty}m_{-\vartheta,\,\lambda}(x_1)\cdot w_\lambda(\vartheta)\phi^{\uparrow}_{\vartheta,\,\lambda}(x_2)\phi^{\downarrow}_{\vartheta,\,\lambda}(r) d\vartheta\/.
  \end{eqnarray*}
\end{proof}
\textbf{Remark 1.} The result in the case $x_2>r\geq1$ can be obtained from the one given above by interchanging the role of $x_2$ and $r$ in the integral on the right-hand side of (\ref{Interval_formula}). This is an easy consequence of the symmetry of the Green function (\ref{Diffusion_Green}) with respect to the speed measure $m(dr)$. However, the result given in (\ref{Interval_formula}) includes the most important case $x_2=1$ which corresponds to $z_1=0$.

\textbf{Remark 2.} The equations (\ref{IntervalFirstEq}) and (\ref{IntervalSecondEq}) can be reduced to the \textit{spheroidal wave equation}
\begin{equation}
\label{SpheroidalEquation}
   (1-z^2)y''(z)-2zy(z)+[\Lambda_{\nu,\mu}(\gamma)-\frac{\mu^2}{(1-z^2)}+\gamma^2(1-z^2)]y(z) = 0\/.
\end{equation}
The \textit{radial spheroidal functions} $S_{\nu,\mu}^{(1)}(\gamma,z)$ and  $S_{\nu,\mu}^{(2)}(\gamma,z)$ and \textit{angular spheroidal functions}  $PS_{\nu,\mu}(\gamma,z)$ and  $QS_{\nu,\mu}(\gamma,z)$ are solutions to the spheroidal wave equation in appropriate regions (see \cite{Robin:1959}). 
When $\gamma^2=0$ the equation (\ref{SpheroidalEquation}) reduces to the Legendre equations (\ref{LegendreDE}). When $\mu=1/2$ (i.e. $\alpha=1$) the equation (\ref{SpheroidalEquation}) can be reduced to the \textit{Mathieu equation} and the spheroidal wave functions are reduced to the \textit{Mathieu functions} (see \cite{McLachlan:1964}, \cite{AbramowitzStegun:1972}).

To illustrate our method we compute what appears to be the Poisson kernel of the interval $[-1,1]$ for the standard
isotropic $\alpha$-stable process (see Proposition \ref{m_harmonic}). 

\begin{cor}
 For $\lambda=0$, $z_1=0$, $|z_2|<1$ and $r>1$ we have
 \begin{eqnarray*}
 \E^{(0,z_2)}\[B(\tau_{C_1})\in dr\] &=& \frac{\sin(\pi\alpha/2)}{\pi}\(\frac{1-z_2^2}{r^2-1}\)^{\alpha/2} \frac{1}{r-z_2}\/.
 \end{eqnarray*}
\end{cor}
\begin{proof}
For $\lambda=0$ the equation (\ref{IntervalFirstEq}) becomes  
\begin{equation*}
  (1-x^2)y''(x)-(2-\alpha) x\, y '(x)-2\vartheta y(x) =0\/.
\end{equation*}
Making substitution $y(x) = u(t)$ with $t=(x+1)/2$ we get 
\begin{eqnarray}
   \label{hyp_equation}
   t(1-t)u''(t) - [1-\frac{\alpha}{2}-(2-\alpha)t]u'(t) -2\vartheta u(t) = 0\/,
\end{eqnarray} 
which is the hypergeometric equation (\ref{HyperDE}) with $c=1-\frac{\alpha}{2}$ and $a = \frac12(1-\alpha)+\frac12((1-\alpha)^2-8\vartheta)^{1/2}$, $b = \frac12(1-\alpha)-\frac12((1-\alpha)^2-8\vartheta)^{1/2}$. Taking into account the general solution of (\ref{hyp_equation}) and the boundary conditions $m_\vartheta(-1)=0$ and $m_\vartheta(1)=1$ together with (\ref{HyperOne}) we get the following formula for $m_{\vartheta,\,0}$
\begin{equation}
   \label{m_zero}
   m_\vartheta(x) = \(\frac{1+x}{2}\)^{\alpha/2}\frac{\Gamma(\alpha/2+A(\vartheta))\Gamma(1+\alpha/2-A(\vartheta))}{\Gamma(\alpha/2)\Gamma(1+\alpha/2)}\,_2F_1(A(\vartheta),1-A(\vartheta);1+\alpha/2;\frac{1+x}{2})\/,
\end{equation}
where 
\begin{equation*}
  A(\vartheta) = \frac12 +\frac12((1-\alpha)^2-8\vartheta)^{1/2}\/. 
\end{equation*}
Substituting $y(r) = (r^2-1)^{\alpha/4}w(r)$ in (\ref{IntervalSecondEq}) with $\lambda=0$ we arrive at 
the Legendre differential equation (\ref{LegendreDE}) with $\mu=\alpha/2$ and $\nu=B(\vartheta) = -\frac12 +\frac12((1-\alpha)^2+8\vartheta)^{1/2}$. Consequently, the monotone solutions can be chosen as
\begin{eqnarray*}
   \phi^{\uparrow}_{\vartheta,\,0}(r) &=& (r^2-1)^{\alpha/4}P_{B(\vartheta)}^{\alpha/2}(r)\/,\\
   \phi^{\downarrow}_{\vartheta,\,0}(r) &=& (r^2-1)^{\alpha/4}Q_{B(\vartheta)}^{\alpha/2}(r)\/.
\end{eqnarray*}
From (\ref{LegendreFirst}) we get that
\begin{equation*}
   \phi^{\uparrow}_{\vartheta,\,0}(1) = \frac{2^{\alpha/2}}{\Gamma(1-\alpha/2)}
\end{equation*}
and 
\begin{eqnarray*}
   \phi^{\uparrow}_{\vartheta,\,0}(1)\phi^{\downarrow}_{\vartheta,\,0}(r) = \frac{2^{\alpha/2}}{\Gamma(1-\alpha/2)}(r^2-1)^{\alpha/4}Q_{B(\vartheta)}^{\alpha/2}(r)\/.
\end{eqnarray*} 
To evaluate the integral appearing in the right-hand side of (\ref{Interval_formula}) we integrate the function 
\begin{equation}
\label{complex_function}
  \vartheta \rightarrow m_{-\vartheta}(x_1)w_0(\vartheta)\phi^{\uparrow}_{\vartheta,\,0}(1)\phi^{\downarrow}_{\vartheta,\,0}(r)
\end{equation}  
   over the following contour of integration 
 \begin{center}
    \begin{picture}(406,206)(0,0)
      \put(0,103){\vector(1,0){406}} 
      \put(30,0){\vector(0,1){206}}  
      \linethickness{1pt}
      \put(50,15){\line(0,1){176}}  
      \put(50,191){\line(1,0){176}} 
      \put(226,15){\line(0,1){176}} 
      \put(50,15){\line(1,0){176}}  
      \put(70,103){\circle*{3}} \put(68,111){$_{\vartheta_0}$}
      \put(110,103){\circle*{3}}\put(108,111){$_{\vartheta_1}$}
       \put(190,103){\circle*{3}}\put(188,111){$_{\vartheta_2}$}
       \put(350,103){\circle*{3}}\put(348,111){$_{\vartheta_3}$}
       \put(218,96){$_R$}
       \put(42,196){$_{\frac{\alpha}{4}+iR}$}
       \put(42,8){$_{\frac{\alpha}{4}-iR}$}
     \end{picture}
     \end{center}
The function of complex variable $\vartheta \to m_{-\vartheta}(x)$ is meromorphic in the half-space $\Re(\vartheta)>0$ with poles at points $\vartheta_n$ such that
\begin{equation*}
    1+\alpha/2-A(-\vartheta_n) = -n
\end{equation*}
what gives $\vartheta_n = \frac{1}{2}(\alpha+n)(1+n)$.
We have 
\begin{eqnarray*}
   \textrm{Res}\,_{\vartheta_n}(\Gamma(1+\alpha/2-A(-\vartheta))) &=& \lim_{\vartheta\to\vartheta_n}(\vartheta-\vartheta_n)\Gamma(1+\alpha/2-A(-\vartheta))\\
   &=& \lim_{\vartheta\to\vartheta_n}\frac{(\vartheta-\vartheta_n)\pi}{\Gamma(-\alpha/2+A(-\vartheta))\sin\pi(1+\alpha/2-A(-\vartheta))}\\
   &=& (-1)^{n}\frac{1+\alpha+2n}{2\Gamma(n+1)}\/.
\end{eqnarray*}
Consequently
\begin{eqnarray*}
  &&\textrm{Res}\,_{\vartheta_n}(m_{-\vartheta}(x) w_0(\vartheta)\phi^{\uparrow}_{\vartheta,\,0}(1)\phi^{\downarrow}_{\vartheta,\,0}(r)) =\\
&&\(\frac{1+x}{2}\)^{\alpha/2}\frac{(-1)^{n}\Gamma(1+\alpha+n)\,_2F_1(1+\frac{\alpha}{2}+n,-n-\frac{\alpha}{2};1+\frac{\alpha}{2};\frac{1+x}{2})}{2\Gamma(n+1)\Gamma(\alpha/2)\Gamma(1+\alpha/2)(1+\alpha+2n)^{-1}} w_0(\vartheta_n)\phi^{\uparrow}_{\vartheta_n,\,0}(1)\phi^{\downarrow}_{\vartheta_n,\,0}(r)\/.
\end{eqnarray*}
Moreover we have (\cite{Erdelyi:1954} p. 105 2.9.(18))
\begin{eqnarray*}
\(\frac{1+x}{2}\)^{\alpha/2}\lefteqn{\,_2F_1(1+\alpha/2+n,-n-\alpha/2;1+\alpha/2;\frac{1+x}{2}) =}\\
&=&  \(\frac{1-x^2}{4}\)^{\alpha/2}\,_2F_1(-n,n+1+\alpha;1+\alpha/2;\frac{1+x}{2})\\
&=& \frac{\Gamma(n+1)\Gamma(\alpha+1)}{\Gamma(n+1+\alpha)}  \(\frac{1-x^2}{4}\)^{\alpha/2} C_n^{(\rho)}(-x)\\
&=& \frac{\Gamma(n+1)\Gamma(\alpha+1)}{\Gamma(n+1+\alpha)}  \(\frac{1-x^2}{4}\)^{\alpha/2} (-1)^n C_n^{(\rho)}(x)\/,
\end{eqnarray*}
where $\rho=(1+\alpha)/2$ and $C_n^{(\rho)}$ is the Gegenbauer polynomial. We also have $B(\vartheta_n)=n+\alpha/2$ and the Wronskian of the pair $(\phi^{\uparrow}_{\vartheta_n,\,0}, \phi^{\downarrow}_{\vartheta_n,\,0})$ is given by
\begin{eqnarray*}
   W\{\phi^{\uparrow}_{\vartheta_n,\,0}(r),\phi^{\downarrow}_{\vartheta_n,\,0}(r)\} &=& (r^2-1)^{\alpha/2}W\{P_{n+\alpha/2}^{\alpha/2}(r),Q_{n+\alpha/2}^{\alpha/2}(r)\}\\
   &=& (r^2-1)^{\alpha/2-1}e^{i\alpha\pi/2}2^{\alpha}\frac{\Gamma(\frac{n+\alpha}{2}+1)\Gamma(\frac{n+\alpha+1}{2})}{\Gamma(1+\frac{n}{2})\Gamma(\frac{n+1}{2})}\\
   &=& (r^2-1)^{\alpha/2-1}e^{i\alpha\pi/2}\frac{\Gamma(n+\alpha+1)}{\Gamma({n+1})}\/.
\end{eqnarray*}
Thus
\begin{eqnarray*}
   w_0(\vartheta_n)\phi^{\uparrow}_{\vartheta_n,\,0}(1)\phi^{\downarrow}_{\vartheta_n,\,0}(r) &=& \frac{2}{(r^2-1)^{\alpha/2-1} W\{\phi^{\uparrow}_{\vartheta_n,\,0}(r),\phi^{\downarrow}_{\vartheta_n,\,0}(r)\}} \phi^{\uparrow}_{\vartheta_n,\,0}(1)\phi^{\downarrow}_{\vartheta_n,\,0}(r)\\
   &=& e^{-i\alpha\pi/2}\frac{\Gamma({n+1})}{\Gamma(n+\alpha+1)} \frac{2^{\alpha/2}(r^2-1)^{\alpha/4}}{\Gamma(1-\alpha/2)}Q_{n+\alpha/2}^{\alpha/2}(r)\/.
\end{eqnarray*}
Combining all above we conclude that the residuum of the function (\ref{complex_function}) at point $\vartheta_n$ is equal to
\begin{eqnarray*} &&e^{-i\alpha\pi/2}\frac{\Gamma(\alpha+1)(1+\alpha+2n)}{\Gamma(\alpha/2)\Gamma(1+\alpha/2)\Gamma(1-\alpha/2)}\(\frac{1-x_1^2}{2}\)^{\alpha/2}C_n^{(\rho)}(x_2)\frac{\Gamma({n+1})}{\Gamma(n+\alpha+1)} (r^2-1)^{\alpha/4}Q_{n+\alpha/2}^{\alpha/2}(r) \\
&=& e^{-i\alpha\pi/2}\frac{\sin(\pi\alpha/2)}{\pi}\frac{2^{\alpha/2}\Gamma(\frac{1+\alpha}{2})}{\sqrt{\pi}} (1-x_1^2)^{\alpha/2}C_n^{(\rho)}(x_2)\frac{\Gamma({n+1})}{\Gamma(n+\alpha+1)} (r^2-1)^{\alpha/4}Q_{n+\alpha/2}^{\alpha/2}(r)\/.
\end{eqnarray*}
Letting $R\to\infty$ and using the asymptotic expansions of the considered functions we obtain that the integral of the function (\ref{complex_function}) over the line $(\alpha/4-i\infty, \alpha/4+i\infty)$ is equal to the sum of the residues at points $\vartheta_n$.  Details are left to the reader. Consequently, using (\ref{Interval_formula}) we get that $\E^{(0,z_2)}\[B(\tau_{C_1})\in dr\]$ is equal to
\begin{eqnarray*}
    &&\frac{\sin(\pi\alpha/2)}{\pi }\frac{(1-z_2^2)^{\alpha/2}}{(r^2-1)^{\alpha/2}}\sum_{n=0}^\infty e^{-\frac{i\alpha\pi}{2}}\frac{(1+\alpha+2n)2^{\frac{\alpha}{2}}\Gamma(\frac{1+\alpha}{2})}{\sqrt{\pi}}\frac{\Gamma({n+1})}{\Gamma(n+\alpha+1)} (r^2-1)^{\alpha/4}Q_{n+\alpha/2}^{\alpha/2}(r)C_n^{\,(\rho)}(z_2)
\end{eqnarray*}
The relation (\cite{GradsteinRyzhik:2007} 7.312 (1))
\begin{eqnarray*}
   \int_{-1}^1 \frac{(1-x^2)^{\alpha/2}}{r-x}C_n^{(\rho)}(x)\,dx &=& \frac{\sqrt{\pi}2^{1-\alpha/2}}{\Gamma(\frac{1+\alpha}{2})}e^{-\frac{\alpha\pi i}{2}} (r^2-1)^{\alpha/4}Q_{n+\alpha/2}^{\alpha/2}(r)
\end{eqnarray*}
and the orthogonal relations for the Gegenbauer polynomials
\begin{eqnarray*}
   \int_{-1}^1 (1-x^2)^{\alpha/2}C_n^{(\rho)}(x)C_m^{(\rho)}(x)dx = \delta_n(m)2^{-\alpha}\frac{\pi\Gamma(n+1+\alpha)}{(n+\frac{1+\alpha}{2})\Gamma^2(\frac{1+\alpha}{2})\Gamma(n+1)}
\end{eqnarray*}
 give
\begin{eqnarray*}
   \frac{1}{r-x} = \sum_{n=0}^\infty a_n(r) C_n^{(\rho)}(x)\/,
\end{eqnarray*}
where
\begin{eqnarray*}
  a_n(r)&=& e^{-\frac{\alpha\pi i}{2}}\frac{\sqrt{\pi}2^{1-\alpha/2}}{\Gamma(\frac{1+\alpha}{2})} (r^2-1)^{\alpha/4}Q_{n+\alpha/2}^{\alpha/2}(r) \cdot 2^{\alpha}\frac{(n+\frac{1+\alpha}{2})\Gamma^2(\frac{1+\alpha}{2})\Gamma(n+1)}{\pi\Gamma(n+1+\alpha)}\\
  &=& e^{-\frac{\alpha\pi i}{2}}\frac{(2n+\alpha+1)2^{\alpha/2}\Gamma(\frac{1+\alpha}{2})}{\sqrt{\pi}}(r^2-1)^{\alpha/4}\frac{\Gamma(n+1)}{\Gamma(n+1+\alpha)}Q_{n+\alpha/2}^{\alpha/2}(r)\/.
\end{eqnarray*}
Using the above-mentioned formula we get
\begin{eqnarray*}
   \E^{(0,z_2)}\[B(\tau_{C_1})\in dr\] &=&  \frac{\sin(\pi\alpha/2)}{\pi }\frac{(1-z_2^2)^{\alpha/2}}{(r^2-1)^{\alpha/2}}\frac{1}{r-z_2}\/.
\end{eqnarray*}\end{proof}

\section{Hitting distributions in $\R^{n+1}$}

The aim of this section is to generalize the results of Section 4 to higher dimensions. Assume that $0< \alpha<2$.
Let  $\textbf{Y}(t) = (Y_1(t),B^{n}(t))$ be the Brownian-Bessel diffusion in $\R^{n+1}$ defined in Section \ref{rel}. That is $Y_1(t)$ is a $BES^{(-\alpha/2)}$ process independent of  $n$-dimensional Brownian motion 
$B^{n}(t) = (B_2(t),B_3(t),\ldots,B_{n+1}(t))$.  

 We consider the $(n+1)$-dimensional process $\textbf{Y}(t)$ exiting from various open sets which are complements of lower dimensional subsets of  $\R^{n+1}$. We define  
  \begin{eqnarray*}
      D_{n} &=& \{y\in \R^{n+1}: (y_1=0)\wedge (y_2>0)\}^{c}\/,\\
      C_{n} &=& \{y\in \R^{n+1}: (y_1=0)\wedge (|y_2|>1)\}^{c}\/,\\
      H_{n} &=& \{y\in \R^{n+1}: (y_1=0)\wedge  (y_2=0)\}^{c}\/.
  \end{eqnarray*}
  Observe that the above sets are complements of an $n$-dimensional halfspace, an $n$-dimensional strip and an $(n-1)$-dimensional linear subspace, respectively. 
  
  Throughout the whole section we use the following notation. For a point $x=(x_1,x_2,\ldots,x_{n+1})\in\R^{n+1}$ we  denote by $\tilde{x} = (x_2,\ldots,x_{n+1})$  its projection onto $\R^n$, and by $\bar{x} = (x_3,\ldots,x_{n+1})$ its projection onto $\R^{n-1}$. We also denote points from  $\R^n$ by $\tilde{x}$, $\tilde{y}$, etc. Likewise points from $\R^{n-1}$ are denoted by $\bar{x}$, $\bar{y}$, etc.
  For any  points $y=(y_1,y_2,\ldots,y_{n+1})\in\R^{n+1}$, $\tilde{x}=(x_2,\ldots,x_{n+1})\in\R^n$ and $\bar{x}=(x_3,\ldots,x_{n+1})\in\R^{n-1}$ we  denote by 
  \begin{equation*}
      |y-\tilde{x}| =\left(y_1^2 + \sum_{i=2}^{n+1}|y_i-x_i|^2\right)^{1/2},\ 
       |y-\bar{x}| =\left(y_1^2 + y_2^2+ \sum_{i=3}^{n+1}|y_i-x_i|^2\right)^{1/2}
  \end{equation*}
  the Euclidean distance between the point $y\in\R^{n+1}$ and $(0,\tilde{x})\in\R^{n+1}$ and $y\in\R^{n+1}$ and $(0,0,\bar{x})\in\R^{n+1}$, respectively. 
  
   We define the first exit time of $\textbf{Y}$ from $D_n$ by
 \begin{equation*}
      \tau_{D_n} = \inf \{t>0: \textbf{Y}(t)\notin D_n\}\/.
  \end{equation*}  

  We begin with providing the formula for the joint density of $(\tau_{D_n}, \textbf{Y}(\tau_{D_n}))$ when $\textbf{Y}$ starts from the point $y=(y_1,\ldots,y_{n+1})$ such that $y_1=0$.
\begin{thm}
\label{Time_Place_Cor}
For $y=(0,y_2,\ldots,y_{n+1})$ such that $y_2<0$ we have
  \begin{equation}
    \label{Time_Place}
   E^{y}[\tau_{D_n}\in dt, \textbf{Y}(\tau_{D_n}) \in d\tilde{\sigma}]=
 \frac{\sin(\pi\alpha/2)}{2^{n/2}\pi^{1+n/2}}\(\frac{-y_2}{\sigma_2}\)^{\alpha/2}\frac{1}{ t^{1+n/2}}\exp\({-\frac{|\tilde{\sigma}-y|^2}{2t}}\)\/,
 \end{equation}
 where $t>0$ and $\tilde{\sigma}=(\sigma_2,\ldots,\sigma_{n+1})\in\R^n$, $\sigma_2>0$.
\end{thm}  
  \begin{proof}
   We begin with the case $n=1$. Then applying the following formula for the Laplace inverse transform (see \cite{BorodinSalminen:2002} formula 2. page 650) 
   \begin{equation*}
      \mathcal{L}^{-1}_\gamma \(e^{-a\sqrt{\gamma}}\)(y) = \frac{a}{2\sqrt{\pi}y^{3/2}}\exp\(-\frac{a^2}{4y}\)\/,\quad a>0\/,
   \end{equation*}
   to the formula (\ref{Form01_HL}) with $\lambda$ replaced by $\sqrt{2\gamma}$ we obtain
   \begin{eqnarray*}
       f^{y_2}(t,\sigma_2) &\stackrel{\textrm{def}}{=}& 
      E^{(0,y_2)}\[\tau_{D_1} \in dt, Y(\tau_{D_1})\in d\sigma_2 \]\\
      &=& E^{(0,y_2)}\[\tau_{D_1} \in dt, W_2(\tau_{D_1})\in d\sigma_2 \]\\
       &=& \frac{\sin(\pi\alpha/2)}{\pi}\(\frac{-y_2}{\sigma_2}\)^{\alpha/2}\mathcal{L}^{-1}_\gamma \( \frac{e^{-\sqrt{2\gamma}(\sigma_2-y_2)}}{\sigma_2-y_2}\)(t)\\
      &=& \frac{\sin(\pi\alpha/2)}{2^{1/2}\pi^{3/2}}\(\frac{-y_2}{\sigma_2}\)^{\alpha/2}\frac{1}{t^{3/2}}
    \exp\({-\frac{(\sigma_2-y_2)^2}{2t}}\)\/.
   \end{eqnarray*}
   This ends the proof for $n=1$. When $n\geq 2$ we observe that the density of the joint distribution of $(\tau_{D_n}, B_2(\tau_{D_n}))$ with respect to $P^{(0,y_2)}$  is equal to $f^{y_2}(t,\sigma_2)$. Let $g_t(\bar{x})$ denote the density function of the distribution of $(n-1)$-dimensional Brownian motion. Obviously, we have 
   \begin{equation*} 
        g_t(\bar{x}) = \frac{1}{(2\pi t)^{\frac{n-1}{2}}}\exp\(-\frac{|\bar{x}|^2}{2t}\)\/.
   \end{equation*}
   Using the fact that the first exit time $\tau_{D_n}$ depends only on $Y_1$ and $B_2$ we obtain
   \begin{eqnarray*}
      E^{y}[\tau_{D_n}\in dt,\textbf{Y}(\tau_{D_n}) \in d\tilde{\sigma}]
&=&E^{y}[\tau_{D_n}\in dt,{B_2}(\tau_{D_n}) \in d\sigma_2,{B_3}(\tau_{D_n}) \in d\sigma_3,\dots {B_{n+1}}(\tau_{D_n}) \in d\sigma_{n+1}]\\
&=&E^{(0,y_2)}[\tau_{D_n}\in dt,{B_2}(\tau_{D_n}) \in d\sigma_2]\,g_{t}(\bar{y}-\bar{\sigma})\\
&=&f^{y_2}(t,\sigma_2)\,g_{t}(\bar{y}-\bar{\sigma})\\
&=&\frac{\sin(\pi\alpha/2)}{2^{n/2}\pi^{1+n/2}}\(\frac{-y_2}{\sigma_2}\)^{\alpha/2}\frac{1}{ t^{1+n/2}}\exp\({-\frac{|\tilde{\sigma}-y|^2}{2t}}\)\,. 
   \end{eqnarray*} 
  \end{proof}
  Now we present the  multi-dimensional generalization of result given in Theorem \ref{THM_HL2}.
 \begin{thm} \label{THM_HL3}
 For $y=\R^{n+1}$, such that $(y_1,y_2,\ldots,y_{n+1})\in D_n$ we have
 \begin{eqnarray*}
    E^{(y_1,\tilde{y})}[\lefteqn{e^{-\frac{\lambda^2}{2}\tau_{D_n}};\textbf{Y}(\tau_{D_n}) \in d\tilde{\sigma}]= \frac{2y_1^\alpha \lambda^\frac{ n+\alpha}{ 2}}{(2\pi )^{n/2}2^{\alpha/2}\Gamma(\alpha/2)}  \frac{K_\frac{ n+\alpha}{ 2}(\lambda|y-\tilde{\sigma}|)}{|y-\tilde{\sigma}|^\frac{ n+\alpha}{ 2}}\,+}\\
  &&+\, \frac{4\sin(\pi\alpha/2)}{2^{n+\alpha/2}\pi^{n+1}\Gamma(\alpha/2)}\, y_1^\alpha \lambda^{n+\alpha/2}\int_{(-\infty,0)\times\R^{n-1}}  \, \(\frac{-z_2}{\tilde{\sigma}_2}\)^{\frac{\alpha}{2}}\frac{K_{\frac{n+\alpha}{2}}(\lambda|y-\tilde{z}|)}{|y-\tilde{z}|^{\frac{n+\alpha}{2}}}
  \frac{K_{n/2}(\lambda|\tilde{\sigma}-\tilde{z}|) }{|\tilde{\sigma}-\tilde{z}|^{n/2}}d\tilde{z}\/,
 \end{eqnarray*}
 where $\tilde{\sigma}=(\sigma_2,\ldots,\sigma_{n+1})\in \R^n$, $\sigma_2>0$.
 For $y_1=0$ we get
 \begin{equation}\label{Laplace99}
   E^{y}[e^{-\frac{\lambda^2}{2}\tau_{D_n}};\textbf{Y}(\tau_{D_n}) \in d\tilde{\sigma}]= \frac{2\sin(\pi\alpha/2)\lambda^{n/2}}{2^{\frac{n}{2}}\pi^{\frac{n+2}{2}}}  \, \(\frac{-y_2}{\sigma_2}\)^{\alpha/2}
   \frac{K_{n/2}(\lambda|y-\tilde{\sigma}|) }{|y-\tilde{\sigma}|^{n/2}} \/.
   \end{equation} 
   \end{thm}
\begin{proof} 
   For $y=(y_1,y_2,\ldots,y_{n+1})\in D_n$ such that $y_1=0$ we use Theorem \ref{Time_Place_Cor} and formula (\ref{Macdonald}) to show that
   \begin{eqnarray*}
      H_\lambda(\tilde{y},\tilde{\sigma}) &=& E^{y}[e^{-\frac{\lambda^2}{2}\tau_{D_n}};\textbf{Y}(\tau_{D_n}) \in d\tilde{\sigma}]\\
&=&\frac{\sin(\pi\alpha/2)}{2^{\frac{n}{2}}\pi^{\frac{n+2}{2}}}\(\frac{-y_2}{\sigma_2}\)^{\alpha/2}\int_0^\infty e^{-\frac{\lambda^2}{2}s}\frac{1}{ s^{1+n/2}}   e^{-\frac{|\tilde{\sigma}-y|^2}{2s}}ds\\
 &=& \frac{2\sin(\pi\alpha/2)}{2^{\frac{n}{2}}\pi^{\frac{n+2}{2}}} \lambda^{n/2}\, \(\frac{-y_2}{\sigma_2}\)^{\alpha/2}\frac{   K_{n/2}(\lambda|\tilde{\sigma}-y|)}{  |\tilde{\sigma}-y|^{n/2}}  \,. 
   \end{eqnarray*} 
   
   For the general point $y=(y_1,y_2,\ldots,y_{n+1}) = (y_1,\tilde{y})\in \R^{n+1}$ we define the first hitting time of $0$ by the process $Y_1$, 
   \begin{equation*}
      T = \inf\{t > 0: Y_1(t) = 0\}\/.
  \end{equation*}
   We denote by 
  $h(y_1,t)$ 
    the density of $T$.
  This  density was found by Getoor and  Sharpe (see \cite{GetoorSharpe:1979}) and is given by 
  \begin{equation*}
  h(y_1,t)= \frac{y_1^\alpha}{2^{\alpha/2}\Gamma(\alpha/2)t^{1+\alpha/2}}\,
  \exp \(-\frac{y_1^2}{2t} \), \ t>0\/.
  \end{equation*}
  Using the independence of $T$ and $B^n$ and (\ref{Macdonald}) we easily obtain
   \begin{eqnarray*}
       P_\lambda(y,\tilde{z}) &=& E^{(y_1,\tilde{y})}\[e^{-\frac{\lambda^2}{2}T}; B^n(T) \in d\tilde{z}\]\\
       &=& \int_0^\infty e^{-\frac{\lambda^2}{2}t} \frac{1}{(2\pi t)^{n/2}}\exp({-\frac{|\tilde{y}-\tilde{z}|^2}{2t}})h(y_1,t)dt\\
       &=& \frac{y_1^\alpha}{(2\pi )^{n/2}2^{\alpha/2}\Gamma(\alpha/2)}\int_0^\infty e^{-\frac{\lambda^2}{2}t} \frac{1}{ t^{1+(n+\alpha)/2}}\exp({-\frac{|y-\tilde{z}|^2}{2t}})dt\\
       &=&\frac{2y_1^\alpha \lambda^\frac{ n+\alpha}{ 2}}{(2\pi )^{n/2}2^{\alpha/2}\Gamma(\alpha/2)}  \frac{K_\frac{ n+\alpha}{ 2}(\lambda|y-\tilde{z}|)}{|y-\tilde{z}|^\frac{ n+\alpha}{ 2}}\/.
 \end{eqnarray*} 
   Using the strong Markov Property we obtain
   \begin{eqnarray*}
      \lefteqn{E^{y}\[e^{-\frac{\lambda^2}{2}\tau_{D_n}};\textbf{Y}(\tau_{D_n})\in d\tilde{\sigma}\]}\\ 
    &=& E^{y}\[T<\tau_{D_n};e^{-\frac{\lambda^2}{2}\tau_{D_n}};\textbf{Y}(\tau_{D_n})\in d\tilde{\sigma})\] +E^{y}\[T=\tau_{D_n};e^{-\frac{\lambda^2}{2}\tau_{D_n}}; \textbf{Y}(\tau_{D_n})\in d\tilde{\sigma}\]\\
    &=& E^{y}\[B_2(T)<0;e^{-\frac{\lambda^2}{2}T}E^{(0,B^n(T))}\[e^{-\frac{\lambda^2}{2}\tau_{D_n}};\textbf{Y}(\tau_{D_n})\in d\tilde{\sigma}\]\] +E^{y}\[B_2(T)\geq0;e^{-\frac{\lambda^2}{2}T};\textbf{Y}({T})\in d\tilde{\sigma}\]\/.
 \end{eqnarray*}
  Consequently we get the following expression for $E^{y}[e^{-\frac{\lambda^2}{2}\tau_{D_n}};\textbf{Y}(\tau_{D_n}) \in d\tilde{\sigma}]$ for every $y\in D_n$.
   \begin{eqnarray*}
   \int_{(-\infty,0]\times\R^{n-1}}P_\lambda(y,\tilde{z})H_\lambda(\tilde{z},\tilde{\sigma})d\tilde{z} + P_\lambda(y,\tilde{\sigma})\/.
   \end{eqnarray*}
   This ends the proof.
   \end{proof} 
   
   The following corollary is an immediate consequence of Theorem \ref{THM_HL3} (formula (\ref{Laplace99})) and Proposition \ref{m_harmonic}.  
   
   \begin{cor}  Let $\tilde{D}_n\subset \R^n$ be  the halfspace $\{\tilde{x}\in R^n;\ x_1<0\}\subset\R^n$. Then the $m$-Poisson kernel of $\tilde{D}_n$ for the relativistic $\alpha$-stable process with parameter $m>0$ is given by 
    \begin{equation*}
   P^m_{\tilde{D}_n}(\tilde{y}, \tilde{\sigma}) = \frac{2\sin(\pi\alpha/2)m^{\frac n{2\alpha}}}{2^{\frac{n}{2}}\pi^{\frac{n+2}{2}}}  \, \(\frac{-y_1}{\sigma_1}\)^{\alpha/2}
   \frac{K_{n/2}(m^{\frac 1{\alpha}}|\tilde{y}-\tilde{\sigma}|) }{|\tilde{y}-\tilde{\sigma}|^{n/2}}\/, 
   \end{equation*}
   where  $\tilde{y}=(y_1,\ldots,y_n) \in \tilde{D}_n$ and $\tilde{\sigma}=(\sigma_1,\ldots,\sigma_n) \in \tilde{D}^c_n$. For $m=0$ we obtain  the  Poisson kernel of  $\tilde{D}_n$ for the standard isotropic  $\alpha$-stable process given by the formula
    \begin{equation*}
   P_{\tilde{D}_n}(\tilde{y}, \bar{\sigma}) = \frac{\sin(\pi\alpha/2)\Gamma(\frac{n}2)} {\pi^{\frac{n+2}{2}}}\(\frac{-y_1}{\sigma_1}\)^{\alpha/2}
   \frac{1} {|\tilde{y}-\bar{\sigma}|^{n}}\/. 
   \end{equation*}
     \end{cor}
     The conclusion of the above  corollary (for $m>0$) is one  the main  results obtained in \cite{BMR:2009}, by using a different approach, which was almost entirely analytical. The present presentation provides another proof which is of probabilistic nature. For $m=0$, i.e. for the  standard isotropic  $\alpha$-stable process,   the formula for its Poisson kernel of $\tilde{D}_n$ is well known and can be obtained from a formula of the Poisson kernel for a ball (see \cite{BGR:1999b}), where an indispensable tool was  Kelvin's transform.

  We define the first exit time of $\textbf{Y}$ from $C_{n}$ by
  \begin{equation*}
      \tau_{C_{n}} = \inf \{t>0: \textbf{Y}(t)\notin C_{n}\}\/.
  \end{equation*} 
 In order to describe the distribution of $(\tau_{C_{n}}, \textbf{Y}(\tau_{C_{n}}))$ 
 we define  $$   H( z_1,z_2, \lambda, r)= \E^{(z_1,z_2)}\[e^{-\frac{\lambda^2}{2}\tau_{C_1}};B_2(\tau_{C_1})\in dr\],\ |r|>1, $$
where $(z_1,z_2)\in C_1$.
The integral  representation of $H( z_1,z_2, \lambda, r)$ is the main result of Section \ref{odcinek} (see the formula (\ref{Interval_formula})).
We define $$ h(\lambda, y,\tilde{\sigma}) = E^{y}[e^{-\frac {\lambda^2}2\tau_{C_{n}}}; B^n(\tau_{C_{n}}) \in d\tilde{\sigma}], \ \tilde{\sigma}= (\sigma_2,\ldots,\sigma_{n+1}) \in \R^n, |\sigma_2|>1\/. $$
In the next theorem we provide a formula for the $(n-1)$-dimensional Fourier transform of $ h(\lambda, y,\tilde{\sigma})$, which entirely describes the  distribution of $(\tau_{C_{n}}, \textbf{Y}(\tau_{C_{n}}))$.
\begin{thm}  
\label{Time_Place_Cor_Strip}Let $n\ge 2$.
For $y=(0,y_2,\ldots,y_{n+1})$ such that $|y_2|<1$ and $ \bar{z}\in \R^{n-1}$ we have 
  \begin{equation}
    \label{Strip}
    \int_{\R^{n-1}}h(\lambda, y,\tilde{\sigma}) e^{i(\bar{\sigma},\bar{z})}d\bar{\sigma}=H( 0,y_2, \sqrt{|\bar{y}-\bar{z}|^2+\lambda^2}, \sigma_2 ).\end{equation}
 Here $\tilde{\sigma}=
(\sigma_2, \bar{\sigma})$.
\end{thm}  
  \begin{proof}
  
  Let $ f^{y_2}(t,\sigma_2)=E^{(0,y_2)}\[\tau_{C_1} \in dt, B_2(\tau_{C_1})\in d\sigma_2 \]$. 
    When $n\geq 2$ we observe that the density of the joint distribution of $(\tau_{C_{n}}, B_2(\tau_{C_{n}}))$ with respect to $P^{(0,y_2)}$  is equal to $f^{y_2}(t,\sigma_2)$. Let $g_t(\bar{x})$ denote the density function of the distribution of $(n-1)$-dimensional Brownian motion.
   %
   Using the fact that the first exit time $\tau_{C_{n}}$ depends only on $Y_1$ and $B_2$ we obtain
   \begin{eqnarray*}
      E^{y}[\tau_{C_{n}}\in dt,\textbf{Y}(\tau_{C_{n}}) \in d\sigma]
&=&E^{y}[\tau_{C_{n}}\in dt,{B_2}(\tau_{C_{n}}) \in d\sigma_2,{B_3}(\tau_{C_{n}}) \in d\sigma_3,\dots {B_{n+1}}(\tau_{C_{n}}) \in d\sigma_{n+1}]\\
&=&E^{(0,y_2)}[\tau_{C_{n}}\in dt,{B_2}(\tau_{C_{n}}) \in d\sigma_2]\,g_{t}(\bar{y}-\bar{\sigma})\\
&=&f^{y_2}(t,\sigma_2)\,g_{t}(\bar{y}-\bar{\sigma})\,. 
   \end{eqnarray*} 
   This implies that 
 $$ h(\lambda, y,\tilde{\sigma}) = E^{y}[e^{-\frac {\lambda^2}2\tau_{C_{n}}}; \textbf{Y}(\tau_{C_{n}}) \in d\tilde{\sigma}]=\int_0^\infty e^{-\frac {\lambda^2}2t} f^{y_2}(t,\sigma_2)\,g_{t}(\bar{y}-\bar{\sigma})dt. $$
 Hence,  the  formula for the $(n-1)$-dimensional Fourier transform of $h(\lambda, y,\tilde{\sigma})$ can easily be  found 
 $$  \int_{\R^{n-1}}h(\lambda, y,\sigma) e^{i(\bar{\sigma},\bar{z})}d\bar{\sigma}=\int_0^\infty e^{-\frac {\lambda^2}2t} f^{y_2}(t,\sigma_2)\,\int_{\R^{n-1}}g_{t}(\bar{y}-\bar{\sigma})e^{i(\bar{\sigma},\bar{z})}d\bar{\sigma}dt=\int_0^\infty f^{y_2}(t,\sigma_2)\,e^{-t\frac{|\bar{y}-\bar{z}|^2+\lambda^2}2}dt. $$
 Next, observe that the last integral is equal to $H( 0,y_2, \sqrt{|\bar{y}-\bar{z}|^2+\lambda^2}, \sigma_2 )$.
 \end{proof}
 
  Again, the following corollary is an immediate consequence of Theorem \ref{Time_Place_Cor_Strip}  and Proposition \ref{m_harmonic}.  
   
   \begin{cor} Assume that $n\ge 2$. Let $\tilde{C}_n\subset \R^n$ be  the strip $\{\tilde{x}\in \R^n;\ |x_1|<1\}\subset \R^n$. Then the $(n-1)$-dimensional Fourier transform of  $m$-Poisson kernel of $\tilde{C}_n$ for the relativistic $\alpha$-stable process with parameter $m\ge 0$ is given by

  \begin{equation*} \int_{\R^{n-1}}P^m_{\tilde{C}_n}(\tilde{y}, \tilde{\sigma}) e^{i(\bar{\sigma},\bar{z})}d\bar{\sigma}
    = H( 0,y_1, \sqrt{|\bar{y}-\bar{z}|^2+m^{2/\alpha}}, \sigma_1 )\ , 
   \end{equation*}
   where  $\tilde{y}=(y_1,\ldots,y_n) \in \tilde{C}_n$ and $\tilde{\sigma}=(\sigma_1,\bar{\sigma}) \in \tilde{C}^c_n$.
     \end{cor}
 
  Note that the above formula for $m=0$ provides the $(n-1)$-dimensional Fourier transform of the Poisson kernel of the strip $\tilde{C}_n$  for the standard isotropic $\alpha$-stable process.

 In the last part of this section we assume that $1<\alpha<2$. We consider a similar problem as above but now we  compute hitting distributions when the $(n+1)$-dimensional  process hits $(n-1)$-dimensional half-space, with $n > 1$.   
 Let 
   \begin{eqnarray*}
       H_{n} &=& \{y\in \R^{n+1}: (y_1=0)\wedge  (y_2=0)\wedge (y_3>0) \}^{c}\/,
  \end{eqnarray*}
 and let $ \tau_{H_{n}}$ be the first exit time of $\textbf{Y}$ from $H_{n}$.
  
  We will reduce the problem to an $n$-dimensional situation. Observe that $Z=\sqrt{Y^2_1+B^2_2}$ is a Bessel process of index $(1-\alpha)/2$ (see \cite{RevuzYor:2005}, Ch.\,XI) so  $\textbf{W}=(Z, B_3,\dots, B_{n+1})$ is our $n$-dimensional Bessel-Brownian diffusion.  Obviously all of its components are independent. Let 
  \begin{eqnarray*}
         \tilde{H}_n &=& \{\tilde{y}\in \R^{n}: (y_1=0)\wedge (y_2>0) \}^{c}\/.
  \end{eqnarray*}
 We define $\tau_{  \tilde{H}_n}$ being the first exit time of $\textbf{W}$ from $  \tilde{H}_n$.
  Observe that 
  $\textbf{Y}$ exits $H_{n}$ if $\textbf{W}$ exits  $  \tilde{H}_n$.
  Moreover 
  
  $$(\tau_{  \tilde{H}_n}, (0,\textbf{W}(\tau_{  \tilde{H}_n})) )= (\tau_{H_{n}}, \textbf{Y}(\tau_{H_{n}}) ).$$
  
Hence our problem is within the framework of the problem studied in the first part of this section (for the process $\textbf{W}$) but with $n$ replaced by $n-1$ and $\alpha$ by $\alpha-1$. Applying  Theorem \ref{THM_HL3} (Theorem \ref{THM_HL2} if $n=2$) we obtain
  \begin{thm} \label{THM_HL4}
 For $y=\R^{n+1}$, such that  $y_1^2+y_2^2>0$ and $1<\alpha<2$ we have
 
 \begin{eqnarray*}
   && E^{y}[\lefteqn{e^{-\frac{\lambda^2}{2}\tau_{H_n}};\textbf{Y}(\tau_{H_n}) \in d\bar{\sigma}]= \frac{2(y_1^2+y_2^2)^{\frac{\alpha-1}2} \lambda^\frac{ n-2+\alpha}{ 2}}{(2\pi )^{\frac{n-1}2}2^{\frac{\alpha-1}2}\Gamma(\frac{\alpha-1}2)}  \frac{K_\frac{ n-2+\alpha}{ 2}(\lambda|y-\bar{\sigma}|)}{|y-\bar{\sigma}|^\frac{ n-2+\alpha}{ 2}}\,+}\\
  &&+\, \frac{4\sin(\pi\frac{\alpha-1}2)(y_1^2+y_2^2)^{\frac{\alpha-1}2} \lambda^{\frac{2n-3+\alpha}2}}{2^{\frac{2n-3+\alpha}2}\pi^{n}\Gamma(\frac{\alpha-1}2)}\, \int_{(-\infty,0)\times\R^{n-2}}  \, \(\frac{-z_3}{\sigma_3}\)^{\frac{\alpha-1}{2}}\frac{K_{\frac{ n-2+\alpha}{ 2}}(\lambda|y-\tilde{z}|)}{|y-\bar{z}|^{\frac{ n-2+\alpha}{ 2}}}
  \frac{K_{\frac{n-1}2}(\lambda|\bar{\sigma}-\bar{z}|) }{|\bar{\sigma}-\bar{z}|^{\frac{n-1}2}}d\bar{z}\/,
 \end{eqnarray*}

 where $\bar{\sigma}=(\sigma_3,\ldots,\sigma_{n+1})\in \R^{n-1}$, $\sigma_3>0$.
 For $y_1=y_2=0$ we get
 \begin{equation*}\label{Laplace1000}
   E^{y}[e^{-\frac{\lambda^2}{2}\tau_{H_n}};\textbf{Y}(\tau_{H_n}) \in d\bar{\sigma}]= \frac{2\sin(\pi(\alpha-1)/2)\lambda^{(n-1)/2}}{2^{\frac{n-1}{2}}\pi^{\frac{n+1}{2}}}  \, \(\frac{-y_3}{\sigma_3}\)^{(\alpha-1)/2}
   \frac{K_{(n-1)/2}(\lambda|y-\bar{\sigma}|) }{|y-\bar{\sigma}|^{(n-1)/2}} \/.
   \end{equation*} 
   
   For $n=2$ the first formula can be  simplified to
   \begin{eqnarray}
  \nonumber 
  &&   E^{(z_1,z_2,z_3)}\lefteqn{\[e^{-\frac{\lambda^2}{2}\tau_{H_2}};B_3(\tau_{H_2})\in dr\] =}\\
  && \frac{(|z|+z_3)^{\frac{\alpha-1}{4}}(|z|-z_3)^{\frac{\alpha-1}{2}}}{2^{\frac{3(\alpha-1)}{4}}\Gamma(\frac{\alpha-1}{2})r^{(\alpha-1)/4}}\int_{\lambda}^\infty e^{-(|z|+r)s}(s^2-\lambda^2)^{\frac{\alpha-1}4}I_{\frac{1-\alpha}2} 
     \(\sqrt{2r}\sqrt{|z|+z_3}\,\sqrt{s^2-\lambda^2}\)ds\/,
     \label{Laplace110}
  \end{eqnarray}
  where $|z|=\sqrt{z_1^2+z_2^2+z_3^2}$ and  $z_1^2+z_2^2>0 $.
   \end{thm}
   
  Again, the following corollary is an immediate consequence of Theorem \ref{THM_HL4}  and Proposition \ref{m_harmonic}.

     \begin{cor} \label{THM_HL5}
      Let $ \tilde{H}_2\subset \R^2$ be the complement of   the half-line  $\{\tilde{x}\in \R^2;\ x_1=0, \ x_2>0 \}\subset\R^2$. Then the $m$-Poisson kernel of $ \tilde{H}_2$ for the relativistic $\alpha$-stable process with parameter $m>0$ and $1<\alpha<2$ is given by 
     \begin{eqnarray*}
  \nonumber
  &&   P^m_{ \tilde{H}_2}(\tilde{y}, r)=\\
  && \frac{(|\tilde{y}|+y_2)^{\frac{\alpha-1}{4}}(|\tilde{y}|-y_2)^{\frac{\alpha-1}{2}}}{2^{\frac{3(\alpha-1)}{4}}\Gamma(\frac{\alpha-1}{2})r^{(\alpha-1)/4}}\int_{m^{\frac1\alpha}}^\infty e^{-(|\tilde{y}|+r)s}(s^2-m^{2/\alpha})^{\frac{\alpha-1}4}I_{\frac{1-\alpha}2} 
     \(\sqrt{2r}\sqrt{|\tilde{y}|+y_2}\,\sqrt{s^2-m^{2/\alpha}}\)ds\/,
  \end{eqnarray*}
 where $r>0$ and    $\tilde{y}=(y_1,y_2) \in  \tilde{H}_2$ with $y_1\neq 0$. 
 
 If $y_1=0$ and $y_2<0$ we have

    \begin{equation*}
   P^m_{ \tilde{H}_2}(\tilde{y}, r) =\frac{\sin(\pi(\alpha-1)/2)}{\pi} \(\frac{-y_2}{r}\)^{(\alpha-1)/2}\frac{e^{{-m^{1/\alpha}}(r-u)}}{r-u}\/.
   \end{equation*}
  For $m=0$ we obtain  the  Poisson kernel of  $\tilde{H}_2$ for the standard isotropic  $\alpha$-stable process given by the formula
  
  \begin{eqnarray}
    P_{ \tilde{H}_2}(\tilde{y}, r)
  = \frac{(|\tilde{y}|+y_2)^{\frac{\alpha-1}{4}}(|\tilde{y}|-y_2)^{\frac{\alpha-1}{2}}}{2^{\frac{3(\alpha-1)}{4}}\Gamma(\frac{\alpha-1}{2})r^{(\alpha-1)/4}}\int_{0}^\infty e^{-(|\tilde{y}|+r)s}s^{(\alpha-1)/2}I_{\frac{1-\alpha}2} 
     \(s\,\sqrt{2r}\sqrt{|\tilde{y}|+y_2}\)ds\/, \label{Poisson100}
  \end{eqnarray}
  where $r>0$ and    $\tilde{y}=(y_1,y_2) \in  \tilde{H}_2$ with $y_1\neq 0$. 
 
 If $y_1=0$ and $y_2<0$ we have
    \begin{equation*}
   P_{ \tilde{H}_2}((0,y_2), r) =\frac{\sin(\pi(\alpha-1)/2)}{\pi} \(\frac{-y_2}{r}\)^{(\alpha-1)/2}\frac{1}{r-u}\/.
   \end{equation*}

     \end{cor}
     
     To our best knowledge the above  formulas have  not been known before even in the stable case. As far as we know the only explicit formula related to the standard isotropic $\alpha$-stable process ($1 <\alpha<2$) killed on exiting $ \tilde{H}_2$ was the formula for the Martin kernel of   $ \tilde{H}_2$ with the pole at infinity  obtained in \cite{BogJak:2005}. 
     
     Next, we present a multidimensional version of Corollary \ref{THM_HL5}.
     For the purpose of  clarity we  give the description of the Poisson kernels if the starting point of the process belongs to the subspace spanned by the underlying half-space. The general case can be easily recovered from Theorem \ref{THM_HL4}. 
       \begin{cor} \label{THM_HL6} Let $\tilde{H}_n\subset \R^n$ be  the complement of the $(n-1)$-dimensional  half-space $\{\tilde{x}\in R^n;\ x_1=0, x_2>0\}\subset\R^n$. Then the $m$-Poisson kernel of $\tilde{H}_n$ for the relativistic $\alpha$-stable process with parameter $m>0$ and $1<\alpha<2$ is given by 
    \begin{equation*}
   P^m_{\tilde{H}_n}(\tilde{y}, \bar{\sigma}) = \frac{2\sin(\pi(\alpha-1)/2)m^{\frac n{2\alpha}}}{2^{\frac{n-1}{2}}\pi^{\frac{n+1}{2}}}  \, \(\frac{-y_2}{\sigma_2}\)^{(\alpha-1)/2}
   \frac{K_{(n-1)/2}(m^{\frac 1{\alpha}}|\tilde{y}-\bar{\sigma}|) }{|\tilde{y}-\bar{\sigma}|^{(n-1)/2}}\/, 
   \end{equation*}
   where  $\tilde{y}=(y_1,\ldots,y_n); y_1=0, y_2<0 $ and $\bar{\sigma}=(\sigma_2,\ldots,\sigma_n); \sigma_2>0 $.
   
    For $m=0$ we obtain  the  Poisson kernel of  $\tilde{H}_n$ for the standard isotropic  $\alpha$-stable process given by the formula
    \begin{equation*}
   P_{\tilde{H}_n}(\tilde{y}, \bar{\sigma}) = \frac{\sin(\pi(\alpha-1)/2)\Gamma(\frac{(n-1)}2)} {\pi^{\frac{n+1}{2}}}\(\frac{-y_2}{\sigma_2}\)^{(\alpha-1)/2}
   \frac{1} {|\tilde{y}-\bar{\sigma}|^{n-1}}\/. 
   \end{equation*}
     \end{cor}
     
     It is very striking that the above formulas are identical with the formulas of Poisson kernels of $(n-1)$-dimensional  half-spaces for the stable process or relativistic stable  with index $\alpha-1$ living in $\R^{n-1}$. But we have to keep in mind that this is true only when the $n$-dimensional process (either $\alpha$-stable or relativistic $\alpha$-stable) starts from the subspace spanned by the complement to $\tilde{H}_n$. If the process starts from other points the formulas become much more complicated (see e.g. (\ref{Poisson100}) for the general two-dimensional stable case).

\section{Appendix}
\subsection{Proof of Proposition \ref{m_harmonic}}

Recall that $\textbf{Y}(t) =(Y_1(t),B^{n}(t))$ is the Bessel-Brownian diffusion in $\R^{n+1}$ defined in Section \ref{rel}. 
We begin with computing the $\lambda$-resolvent kernel  of the semigroup generated by  $\textbf{Y}(t)$. Recall that the transition density  (with respect to the speed measure $m(dy)= 2y^{1-\alpha}dy,\ y\ge 0$ ) of $Y_1(t)$ is given by the following formulas:

\begin{equation}
q_t^{(-\alpha/2)}(x,y) = \frac{1}{2t}
(xy)^{\alpha/2}
  e^{-(x^2+y^2)/2t}I_{-\alpha/2} 
     \(xy/{t}\)
\quad {\text for} \quad x>0
\end{equation}
and
\begin{equation}
q_t^{(-\alpha/2)}(0,y) = 2^{\alpha/2-1}t^{\alpha/2-1}\Gamma(1-\alpha/2)^{-1}\,
  e^{-y^2/2t} \,.
\end{equation}
 For $\alpha=1$, $Y_1 = R^{(-1/2)}$ is the {\it reflected Brownian motion} with the density
\begin{equation} \label{reflectedBM}
q_{t}^{(-1/2)}(x,y) = \frac{1}{\sqrt{2\pi\,t}}\left[ e^{-|x-y|^2 /2t} +  e^{-|x+y|^2 /2t}\right] \,.
\end{equation}

 We  use the following notation: $y=(y_1,\tilde{y})$, where  $\tilde{y}=(y_2,\dots, y_{n+1})\in \R^n$.
   Hence   we find that the transition density of $\textbf{Y}(t)$  with respect to the measure $m(dy_1)\times d\tilde{y}, y_1\in \R^+ , \tilde{y}\in \R^n$ is equal to 

$$q(t;x, y)= q_t^{(-\alpha/2)}(x_1,y_1)g_t(\tilde{x}-\tilde{y}),\ x,y\in \R^{n+1}, $$
where $g_t(\tilde{x}-\tilde{y})$ is the transition density of $B^n$. Observe that $q(t;x, y)$ is symmetric. Let $U_{\lambda}\,, \lambda>0$  denote the $\lambda^2/2$-resolvent kernel for the process $\textbf{Y}(t)$. That is 
$$U_{\lambda}(x, y)= \int_0^\infty e^{-\lambda^2 t/2} q(t;x, y)dt.$$

\begin{lem} \label{potential} 
Suppose that $ x=(0,\tilde{x})$  or $ y=(0,\tilde{y})$ . Then 
\begin{equation}
\label{potentialBessel}
U_{\lambda}(x,y)=\frac{(\lambda/2)^{\frac{n-\alpha}{2}}}{\pi^{n/2}\Gamma(1-\alpha/2)}
\frac {K_{\frac{n-\alpha}2}(\lambda|x-y|)}{|x-y|^{\frac{n-\alpha}{2}}}.
\end{equation}
For $\alpha=1$ we obtain, for all values of $x_1, y_1$:
\begin{equation} \label{refbm}
U_{\lambda}(x,y)=\frac{1}{\pi}\(\frac{\lambda}{2\,\pi}\)^{\frac{n-1}{2}}
\left[
\frac {K_{\frac{n-1}{2}}(\lambda |x-y|)}{|x-y|^{\frac{n-1}{2}}}+
\frac {K_{\frac{n-1}{2}}(\lambda |x-y^{\ast}|)}{|x-y^{\ast}|^{\frac{n-1}{2}}}\right]\,,
\end{equation}
where for $y=(y_1,\tilde{y})$ we have $y^{\ast}=(-y_1,\tilde{y})$.
\end{lem}
\begin{proof}
By symmetry of $U_{\lambda}$ we may assume $x_1=0$. Using (\ref{BESQzero}) we obtain
 \begin{eqnarray*}U_{\lambda}(x,y)&=& \int_0^\infty \frac{(2t)^{\alpha/2-1}}{\Gamma(1-\alpha/2)}\, 
 \frac{e^{-\frac{y_{1}^2}{2t}}}{{(2\pi t)}^{n/2}}\,e^{-\frac {|\tilde{x}-\tilde{y}|^2}{2t}}\,
 e^{-\frac{\lambda^2 t}{2}}\,dt\\
 &=&
 \int_0^\infty \frac{(2t)^{\alpha/2-1}}{\Gamma(1-\alpha/2)} \,\frac{e^{-\frac{|x-y|^2}{2t}}}{{(2\pi t )}^{n/2}}\,
 \,e^{-\frac{\lambda^2 t}{2}}\,dt\\ 
 &=&\frac{(\lambda/2)^{\frac{n-\alpha}{2}}}{\pi^{n/2}\Gamma(1-\alpha/2)}\,
 \frac {K_{\frac{n-\alpha}{2}}(\lambda|x-y|)}{|x-y|^{\frac{n-\alpha}{2}}}. 
 \end{eqnarray*}
For the case $\alpha=1$ we use the formula  \pref{reflectedBM} to obtain \pref{refbm}.
\end{proof}

Now, let $D$ be an open  set of $\R^{n+1}$ and $\tau_D=\inf\{t>0; X_t\notin D\}$. 
A point $y\in  D^c$ is called regular for $D^c$ if 
$$ P^{y}(\tau_D=0)=1.$$
Note that by Blumenthal's $0-1$ law the equivalent condition for regularity is $ P^{y}(\tau_D=0)>0.$ It is clear that all points in the interior of $D^c$ must be regular. Hence regularity must be justified only for points from $\partial{D}$. From the general theory it follows that if $y\in D^c$  is regular for $D$ and $x\in D$ then 
\begin{equation}\label{sweeping}U_\lambda(x, y)= E^{x}[\tau_D<\infty;\, e^{-\lambda \tau_D}U_\lambda(\textbf{Y}(\tau_D),y )].\end{equation}

Next, we consider  sets whose complements are of lower dimensions. Let $F$  be a closed subset of $\R^n$ and take 
$D=   \R^{n+1}\setminus \{0\}\times F.$ 
Observe that $\partial D = D^c=\{0\}\times F$.  We want to impose a condition on the set $F$ which guarantees that a point of  $\partial D$ is regular.  
We say that  $ F\subset \R^n$ has   {\it the interior cone property} at $y\in F$  if there  is an open cone  $\Gamma\subset \R^d$ with vertex at  $y$ and $r>0$ such that $\Gamma_r= \Gamma\cap B(y,r)\subset F$. Note that in the case $n=1$ the above property is equivalent to the condition that for a point $y\in F$ there is an open interval $I_y$ with one of the endpoints equal to $y$ such that $I_y\subset F$.  In particular if we  take  $F= \bigcup_{n} I_n $,  where $I_n$ are disjoint closed intervals, then every point of $F$ has the above property. 
\begin{lem}\label{regular1} Let $y=(0,\tilde{y})\in \{0\}\times F$ and  $F$ has the  interior cone property at $\tilde{y}\in F$. Then the point $y$ is regular for $D^c$.
\end{lem}
\begin{proof}
 We may and do assume that $y=0$.
Let $ \eta_k=\inf\{t>0, Y_1(t)=1/k\}$,
 $T_k= \inf\{t> \eta_k, Y_1(t)=0\}$, where $k=0,1,\dots$. 
 We claim that
 \begin{equation}\label{regular}
 P^0(\lim_{k\to \infty}T_k=0)=1.
 \end{equation}
 %
  %
   Let $\epsilon>0$, then 
  %
 %
  from the scaling property for the Bessel process we have 
  $$ P^{0}(T_k>\epsilon)=P^{0}(T_1>k^2\epsilon).$$
  By the fact that $Y_1$ hits any point of $[0,\infty)$ with probablity one we infer that $ P^{0}(T_1<\infty)=1$, which implies that the decreasing sequence $T_k$ is convergent in $P^0$ probability, hence $P^0$ a.s. 
Observe that on the set $\{\textbf{Y}(T_k)\in \{0\}\times F\}= \{B^n(T_k)\in  F\}$ we have $\tau_D\le  T_k$.
 Hence from (\ref{regular}) we infer that 

$$P^0(\tau_D=0)\ge P^0( \limsup_k \{B^n(T_k)\in  F\})\ge  \limsup_k P^0(B^n(T_k)\in F).$$
Next,
$$ P^0(B^n(T_k)\in F\cap \Gamma)\ge P^0(B^n(T_k)\in  \Gamma_r, T_k\le 1).$$
Noting that $P^0(B^n(t)\in   \Gamma_r)\ge P^0(B^n(1)\in   \Gamma_r)>0,\ 0<t\le 1$, and  using independence of the process $B^n$ and $T_k $ we obtain 
$$ \limsup_k P^0(B^n(T_k)\in F)\ge P^0(B^n_{1}\in   \Gamma_r)P^0({T_k}\le 1)\to P^0(B^n_{1}\in   \Gamma_r)>0.$$
This implies that $P^0(\tau_D=0)>0$ and then $P^0(\tau_D=0)=1$.
  \end{proof}

Let $X^m$ be the $\alpha$-stable  relativistic stable process living in $\R^n$. 
The kernel of the $\lambda$-resolvent of the semigroup generated by $X^m(t)$ will be
 denoted by $U^m_{\lambda}(x,y)$. 
 We have
 \begin{equation*}
 U^m_{\lambda}(x,y) = \int_0^{\infty} e^{-\lambda t}\,p_t^m(x-y)\,dt\,,
 \end{equation*}
where $p_t^m(x-y)$ is the transition density (with respect to the Lebesgue measure)  of the process $X_t^m$  .
 The function $U^m_{\lambda}$ has a particularly
 simple expression when $\lambda=m$ (see \cite{BMR:2009}):
  \begin{equation}\label{m-potential}
  U^m_m(x,y)= \frac{2^{1-(d+\alpha)/2}}{\Gamma(\alpha/2)\pi^{d/2}}
  \frac{m^{\frac{n-\alpha}{2 \alpha}}\,
   K_{(n-\alpha)/2}(m^{1/\alpha}|x-y|)}{|x-y|^{(n-\alpha)/2}}\,.
  \end{equation}

 For open  $\tilde{D}\subset {\R}^{n}$
   we define
   $
   \tau_{\tilde{D}}=\inf\{t> 0;\, X^m(t)\notin \tilde{D}\}\,.
  $
    From the general theory it follows that if $y\in \tilde{D}^c$  is regular for $\tilde{D}^c$ (for the process $X^m$)  and $x\in \tilde{D}$ then 
\begin{equation}\label{sweeping_r}U^m_\lambda(x, y)= E^{x}[\tau_{\tilde{D}}<\infty;\,\,e^{-\lambda \tau_{\tilde{D}}}U^m_\lambda(X^m(\tau_{\tilde{D}}),y )].\end{equation}
Recall that the 
  $m$-{\it harmonic measure}  of the
  set $\tilde{D}$ is defined as
  \begin{equation} \label{harm_def1}
  P_{\tilde{D}}^{m}(x,A)=
  E^x[\tau_{\tilde{D}}<\infty;\,\, e^{-\lambda \tau_{\tilde{D}}} {\bf{1}}_A(X^m(\tau_{\tilde{D}}))],
  \end{equation}
  for a Borel set $A$. If all points of ${\tilde{D}}^c$ are regular then  
 the equation (\ref{sweeping_r}) can be rewritten as 
   \begin{equation} \label{sweeping_out_r}
  \int_{{\tilde{D}}^c} U_{m}^m(z-y)\, P_{\tilde{D}}^{m}(x,dz) =
  U_{m}^m(x-y)\,, \quad x \in \tilde{D}\,, \quad y \in {\tilde{D}}^c\,,
  \end{equation}
  which together with the following uniqueness lemma (see \cite{BMR:2009}), is crucial for obtaining
  the explicit form of $P_{\tilde{D}}^m(x,dz)$ in the particular cases studied in the paper.
  \begin{lem}[Uniqueness] \label{uniqueness}
  Suppose that $\mu$ is a finite signed measure concentrated on a closed set $F\subseteq \R^n$ with the      ({\it finite energy}) property :
  \begin{equation}\label{energy}
  \int_{F}\int_{F} \frac{ K_{(n-\alpha)/2}(m^{1/\alpha}|z-y|)}{|z-y|^{(n-\alpha)/2}} |\mu|(dz)\,|\mu|(dy)\, < \infty\,.
  \end{equation}
  If for every $z\in F$ we have
  \begin{equation}\label{sweeping_out1}
  \int_{F} \frac{ K_{(n-\alpha)/2}(m^{1/\alpha}|z-y|)}{|z-y|^{(n-\alpha)/2}} \,\mu(dy)\, = 0\,,
  \end{equation}
  then $\mu = 0$.
   \end{lem}
 
  Comparing (\ref{m-potential}) with (\ref{potentialBessel}) we see that for  points $x=(0,\tilde {x}), y=(0,\tilde {y})\in \R^{n+1}$ we have the following equality:
  \begin{equation}\label{compare}U^m_m(\tilde {x}, \tilde{ y})=c_\alpha U_{ {m^{1/\alpha}}}(x,y),
  \end{equation}
  where
 $c_\alpha= \frac{\Gamma(1-\frac{\alpha}2)2^{1-\alpha}}{\Gamma(\frac{\alpha}2)}$.
 \begin{lem}\label{m_harmonic2} Let ${\tilde{D}}\subset \R^n$ be open. Assume that $F=\R^n\setminus \tilde{D}$ has the interior cone property at every point. Let $D= \R^{n+1}\setminus (\{0\}\times F )$. Let $x=(0,\tilde {x})\in \{0\}\times \tilde{D}$ . Assume that $P^x(\tau_D<\infty)=1.$ Then the measure  
 $P(x,d\tilde{\sigma})=E^{x} (e^{-\frac {m^{2/\alpha}}2 \tau_D};  B^n(\tau_D)\in d\tilde{\sigma}) $ is the $m$-harmonic measure of $\tilde{D}$  for the $n$-dimensional  $\alpha$-stable relativistic  process with parameter $m>0$.
 \end{lem}
 \begin{proof} By Lemma \ref{regular1} all points of $D^c=\{0\}\times F$ are regular for $D^c$. Then by (\ref{sweeping}) the measure  $P(x,d\tilde{\sigma})$ concentrated on $F$ has to satisfy
 $$U_{ {m^{1/\alpha}}}(x,y)= \int U_{{m^{1/\alpha}}}(\sigma,y)P(x,d\tilde{\sigma}), \ y \in \{0\}\times F.$$
 By (\ref{compare}) this is equvalent to
 \begin{equation}\label{sweeping3}
 U^m_m(\tilde{x},\tilde{y})= \int U^m_m(\tilde{\sigma},\tilde{y})P(x,d\tilde{\sigma}),\ \tilde{y} \in  F.
 \end{equation}
 Integrating both sides with respect to $P(x,d\tilde{y})$ we obtain
 $$\int U^m_m(\tilde{x},\tilde{y})P(x,d\tilde{y})= \int\int\tilde{ U}_{ {m}}(\tilde{\sigma},\tilde{y})P(x,d\tilde{\sigma})P(x,d\tilde{y}).$$
 
 Since $\tilde{x}\in \tilde{D}$, then  $sup_{\tilde{y}\in F} U^m_m(\tilde{x},\tilde{y})<\infty$ and the integral on the left-hand side is finite. This implies that 
 the sub-probability measure $P(x,d\tilde{\sigma})$ has the finite energy (defined in Lemma \ref{uniqueness}).
 Observe that the set $F\subset \R^n$ has the interior cone property, therefore all points of $F$ are regular (for the process $X^m$). The proof of regularity is identical or even simpler  than the proof of Lemma \ref{regular1}. Instead of the sequence of random times $T_k,\ k\ge 1$, one can take a detrministic sequence $1/k$ and use the fact that the process $X^m$ is rotation invariant. This implies that (\ref{sweeping_out_r}) holds and by the same arguments as above, for the measure $P(x,d\tilde{\sigma})$, we infer that the measure $P_{\tilde{D}}^m(\tilde{x},d\tilde{\sigma})$ also has the finite energy.
  Finally we may apply Lemma \ref{uniqueness} together with (\ref{sweeping_out_r}) and (\ref{sweeping3}) to claim that $P(x,d\tilde{\sigma})= P_{\tilde{D}}^m(\tilde{x},d\tilde{\sigma})$.
 \end{proof}
 Observe that the above Lemma provides the proof of Proposition \ref{m_harmonic} for $m>0$. It remains to settle the case $m=0$ and this is done below.
 
 \begin{cor}\label{stable_limit} Let $X(t)$ be the standard isotropic  $\alpha$-stable process in $\R^n$. With the assumptions of  Lemma \ref{m_harmonic2} we have  
   $$P^{\tilde{x}}(X(\tau_{\tilde D})\in A )= P^{x}( B^n(\tau_ D)\in A),$$
   for every Borel set $A\in \R^n$.

 \end{cor}

 \begin{proof}   
It is known that  we can represent the stable process as  $X(t)=X^m(t)+Z^m(t)$, where $Z^m(t)$ is an independent of $X^m(t)$ compound Poisson process with interarrivals of jumps having exponential distribution with parameter $m$ (see \cite{Ryznar:2002}). This implies that there exists $T=T(m)$ with exponential distribution with parameter $m$, independent of $X^m$, such that 

 \begin{equation}\label{equality1} X(t)=X^m(t),\quad t< T.\end{equation} 
 Let    \begin{eqnarray*}\eta&=&\eta(m)=\inf\{t>0, X^m(t)\notin \tilde{D} \}\\
 \tau&=&\inf\{t>0, X(t)\notin \tilde{D} \}.\end{eqnarray*}
  Observe that (\ref{equality1}) implies that  for every Borel set $A\in\R^n$ we have
 
$$ \{X^m(\eta)\in A,\eta < T\}=
 \{X_{\tau}\in A,\tau < T\}.$$          
  Then for every Borel set $A$ we obtain 
 \begin{eqnarray*}
 |E^{\tilde{x}}[ e^{-m\eta},X^m(\eta)\in A]-P^{\tilde{x}}(X(\tau)\in A )|
 &\le& |E^{\tilde{x}}[ e^{-m\eta},X^m(\eta)\in A,\eta< T ]-P^{\tilde{x}}(X(\tau)\in A,\tau < T )|\\
 &+&P^{\tilde{x}}(\eta\ge T(m) )+P^{\tilde{x}}(\tau\ge T(m) )\\
 &=& |E^{\tilde{x}}[ e^{-m\eta},X(\eta)\in A,\eta< T ]-P^{\tilde{x}}(X^m(\eta)\in A,\eta < T )|\\
 &+&2P^{\tilde{x}}(\eta\ge T )\\
 &\le& 3 E^{\tilde{x}}[1- e^{-m\eta}].
 \end{eqnarray*}
 In the last step we used the independence of $\eta$ and  $T$.
 By the relationship between the processes $\textbf{Y}$ and   $X^m $ (see Lemma \ref{m_harmonic2} ) we have 
 $E^{\tilde{x}} e^{-m\eta}= E^{x} e^{-\frac {m^{2/\alpha}}2\tau_ D}$, which yields   $\lim_{m\to 0}E^{\tilde{x}}[1- e^{-m\eta}]=0$. Finally, this implies that 
 
  \begin{eqnarray*}
  P^{\tilde{x}}(X(\tau)\in A )&=& \lim_{m\to 0}E^{\tilde{x}}[ e^{-m\eta},X^m(\eta)\in A]\\
  &=& \lim_{m\to 0} E^{x}[ e^{-\frac {m^{2/\alpha}}2\tau_ D};B^n(\tau_ D)\in A]\\
 &=&  P^{x}( B^n(\tau_ D)\in A).
  \end{eqnarray*}
 \end{proof}


\subsection{Generalized Bessel processes}
In this section we consider the process $X=(X_1,X_2)$ given by the following  stochastic differential equations
 \begin{eqnarray}
    \label{Twolines_SDE01a}
    dX_1 &=& \sqrt{|1-X_1^2|}\,dB_1 -\frac{2-\alpha}{2}X_1\,dt\\
    \label{Twolines_SDE01b}
    dX_2 &=& \sqrt{|X_2^2-1|}\,dB_2 +\frac{2-\alpha}{2}|X_2|\,dt
 \end{eqnarray}
 such that $X_1(0)=x_1$, $X_2(0)=x_2$, where $ |x_1|\leq 1$ and $x_2\geq 1$. Here $B_1$, $B_2$ are two independent standard Brownian motions in $\R$. 
We claim that that for  $|X_{1}(0)|\leq 1$ we have $|X_1(t)|\leq 1$ for all $t>0$.
Analogously, $X_2(t)\geq 1$, whenever 
$X_{2}(0)\geq 1$. 

The first process is a "quadratic" version of {\it Legendre process}; the second one is  "`quadratic" 
{\it hyperbolic Bessel process} (see \cite{RevuzYor:2005}, Ch. VIII, p. 357).
Indeed,  changing variables $X_1=\sin Y_1$ and 
$X_2=\cosh Y_2$ we obtain
\begin{eqnarray}
    \label{Twolines_SDE02a}
    dY_1 &=& dB_1 -\frac{1-\alpha}{2}\tan Y_1\,dt\\
    \label{Twolines_SDE02b}
    dY_2 &=& dB_2 +\frac{1-\alpha}{2}\coth Y_2\,dt
 \end{eqnarray}
with the generators
\begin{eqnarray}
    \label{tangenerator}
    \mathcal{G}_1 &=& \frac{1}{2}\frac{d^2}{dx^2}-\frac{1-\alpha}{2}\tan x\frac{d}{dx}\/,\\
    \label{cotanhgenerator}
    \mathcal{G}_2 &=& \frac{1}{2}\frac{d^2}{dx^2}+\frac{1-\alpha}{2}\coth x\frac{d}{dx}\/,
 \end{eqnarray}
These processes were investigated by various authors, among them by \cite{Gruet:1997} but for different values of coefficients appearing in the drift term. Therefore, we present here, for convenience of the reader, a more detailed information about behaviour of these diffusions. 

Observe that the processes $X_1, X_2$ exist  as {\it solutions admitting explosions} (see
Theorem 2.3 in \cite{IW}, p. 159). Moreover, the corresponding functions $\sigma_i$ and $b_i$
in the equations \pref{Twolines_SDE01a} and  \pref{Twolines_SDE01b}
\begin{equation*}
\sigma_i(X_i(t))\,dB_i(t) + b_i(X_i(t))\,dt\,, \quad i=1, 2
\end{equation*}
satisfy in both cases
\begin{equation*}
|\sigma_i(x)|^2 +|b_i(x)|^2 \leq 2\,(1+|x|^2)\,,
\end{equation*}
which shows that the {\it explosion time} $\zeta=\infty$.

Now, we prove  that for  $|X_{1}(0)|\leq 1$ we have $|X_1(t)|\leq 1$ for all $t>0$ and the boundary points $1$ and $(-1)$ are instantaneously reflecting. Analogously, $X_2(t)\geq 1$, whenever 
$X_{2}(0)\geq 1$, and the boundary point $1$ is also instantaneously reflecting.
We only sketch the first part of the statement. To do this, observe that local times $L_{t}^{1}(X_1)$, $L_{t}^{-1}(X_1)$ are equal $0$:
\begin{equation*}
t-1\geq \int_1^t {\bf 1}_{(1<X_1(s))}\,ds 
= \int_1^t {\bf 1}_{(1<X_1(s))}\,\frac{1}{|1-X_1(s)^2|} d\left\langle X \right\rangle_s
=\int_1^t \frac{1}{|1-a^2|}\,L_t^a(X_1)\,da\,.
\end{equation*}
Thus,  $L_{t}^{1}(X_1)=0$. Analogously,  $L_{t}^{-1}(X_1)=0$. By Tanaka Formula (see \cite{RevuzYor:2005}, Ch. VI)   we have
\begin{equation*}
(X_1(t)-1)^{+}= (X_1(0)-1)^{+} + \int_0^t {\bf 1}_{(1<X_1(s))}\,dX_1(s) +\frac{1}{2}L_{t}^{1}(X_1)\,.
\end{equation*}
Applying $L_{t}^{1}(X_1)=0$ 
we obtain whenever $|X_{1}(0)|\leq 1$ 
\begin{equation*}
(X_1(t)-1)^{+}= \int_0^t {\bf 1}_{(1<X_1(s))}\,dX_1(s) \,.
\end{equation*}
Taking expectations we get
\begin{equation*}
E(X_1(t)-1)^{+}= -\frac{2-\alpha}{2} \int_0^t E[{\bf 1}_{(1<X_1(s))}\,X_1(s)]\,ds \,.
\end{equation*}
Now, the left-hand side is non-negative while the right-hand one is non-positive so both are 
$0$. Consequently, for every $t>0$ we have  $X_1(t)\leq 1$ almost everywhere. The same arguments apply to show that $X_1(t)\geq -1$ almost everywhere.
Thus, the absolute values in equations \pref{Twolines_SDE01a} and  \pref{Twolines_SDE01b} can be discarded.

Another application of local times shows that 
\begin{equation*}
L_{t}^{1}(X_1)-L_{t}^{1-}(X_1) =-2\frac{2-\alpha}{2}\int_0^t {\bf 1}_{(X_1(s)=1)}\,ds
\end{equation*}
Since obviously $L_{t}^{1-}(X_1)=0$ we obtain that that the time spent by $X_1$
in $1$ has zero Lebesgue measure, which shows that the point $1$ is (instantaneously)
reflecting. The same conclusion holds true for the point $(-1)$ and for the point $1$ for the process $X_2$. 

Now, with the information that $|X_1(t)|\leq 1$ for all $t>0$ we can prove the uniqueness 
of the process $X_1(t)$. To show this observe that 
$|\sqrt{1-x^2} - \sqrt{1-y^2}|\leq \sqrt{|x^2-y^2|} = \sqrt{|x+y|}\,\sqrt{|x-y|}$, 
which means that the function $\sigma_1$ is locally H\"older continuous of exponent $1/2$. Moreover, the function $b_1$ is obviously Lipschtiz continuous.  This assures the uniqueness of the process $X_1(t)$ (see remarks after Theorem 6.1, p. 201 in \cite{IW}).  The same observation is true for the process $X_2(t)$. 

 Obviously the processes $X_1$, $X_2$ are independent. Moreover, the generators of the processes are given by
 \begin{eqnarray}
    \label{singenerator_appendix}
    \mathcal{G}_1 &=& \frac{1-x^2}{2}\frac{d^2}{dx^2}-\frac{2-\alpha}{2}x\frac{d}{dx}\/,\\
    \label{coshgenerator_appendix}
    \mathcal{G}_2 &=& \frac{x^2-1}{2}\frac{d^2}{dx^2}+\frac{2-\alpha}{2}x\frac{d}{dx}\/,
 \end{eqnarray}
respectively. We assume that the domain of \pref{singenerator_appendix} consists of the functions
$u \in C^2[-1,1]$ with the property that $u'(-1)=u'(1)= 0$ and, in the  case of \pref{coshgenerator_appendix},
we assume  $u \in C^2[1,\infty)$ along with the property $u'(1)=0$. Derivatives here are meant as (appropriate) one-sided ones.

To classify the boundary points we compute explicitely basic characteristics of the diffusions with generators \pref{singenerator_appendix} and \pref{coshgenerator_appendix}. For both processes we write, in a concise 
form, 
 the scale function $s$  as
\begin{equation*}
s'(x)=\frac{1}{|1-x^2|^{1-\alpha/2}}\,,
\end{equation*}
and the speed measure $m$ as
\begin{equation*}
m(x)=\frac{2}{|1-x^2|^{\alpha/2}}\,.
\end{equation*}
Obviously the killing measure in both cases is $0$. With these explicit formulas at hand we are able to show that all boundary points are 
{\it non-singular} (see \cite{BorodinSalminen:2002}, Ch.\,II), that is, the process reaches each boundary point with positive probability and also starts from this point (at the boundary). We additionally require that the speed measure of each boundary point has the value zero, which is consistent with the {\it reflection} at the boundary. We carry out the appropriate calculation for the point $(-1)$ only. The remaining cases can be calculated in the same way. Now, for $-1<a<z<1$ we have
\begin{eqnarray*}
\int_{-1}^z m(a,z)\,s(da)&=&\int_{-1}^z\left[\int_{a}^z \frac{2\,dx}{(1-x^2)^{\alpha/2}}\right]
\frac{da}{(1-a^2)^{1-\alpha/2}}\\
&\leq&
\int_{-1}^z\left[\int_{a}^z \frac{dx}{(1+x)^{\alpha/2}}\right]
 \frac{da}{(1+a)^{1-\alpha/2}} <\infty\,,
\end{eqnarray*}
which  shows that the point $(-1)$ is the {\it exit}. Similarly, for the same $a$ and $z$ as
above we obtain
\begin{eqnarray*}
\int_{-1}^z (s(z)-s(a))\,m(da)&=&\int_{-1}^z \left[\int_{a}^z \frac{dx}{(1-x^2)^{1-\alpha/2}}\,\right]
 \frac{2\,da}{(1-a^2)^{\alpha/2}}\\
&\leq&
\int_{-1}^z \left[\int_{a}^z \frac{dx}{(1+x)^{1-\alpha/2}}\right]
 \frac{da}{(1+a)^{\alpha/2}}\,<\infty\,,
\end{eqnarray*}
which means that the point $(-1)$ is the {\it entrance}. Therefore, the point $(-1)$ is 
{\it non-singular}. Similar calculations show that the point $1$ is non-singular for 
 the process $X_1$ as well as $X_2$. 
     
          
 \bibliography{bibliography}
\bibliographystyle{plain}

\end{document}